\documentclass{amsart}

\usepackage[utf8]{inputenc}
\usepackage[T1]{fontenc}
\usepackage{lmodern}
\usepackage{amsmath}
\usepackage{amsthm}
\usepackage{amssymb}
\usepackage{units}
\usepackage{amscd}
\usepackage{graphicx}
\usepackage{geometry}
\usepackage{tikz}
\usepackage{tikz-cd}
\usepackage{enumitem}
\usepackage{url}
\usepackage{pgfplots}
\usepackage{hyperref}

{\newtheorem{theointro}{Theorem}

}
{\newtheorem{theo}{Theorem}[section]}
{}
{\theoremstyle{definition} \newtheorem{defin}[theo]{Definition}
							\newtheorem{prop}[theo]{Proposition}
							\newtheorem{lemme}[theo]{Lemma}}
{\theoremstyle{remark} \newtheorem{remarque}[theo]{Remark}
						\newtheorem{exemple}[theo]{Example}}

\newcommand{\Id}{Id}	
\newcommand{\Set}{\text{Set}}
\newcommand{\op}{op}

\newcommand{\sS}{\text{sSet}}
\newcommand{\Hom}{\text{Hom}}

\newcommand{\Top}{\text{Top}}
\newcommand{\Real}[1]{{}||#1||}
\newcommand{\RealP}[1]{{}\Real{#1}_P}
\newcommand{\Sing}{\text{Sing}}

\newcommand{\pr}{pr}
\newcommand{\colim}{\operatornamewithlimits{colim}}

\newcommand{\Fun}{\text{Fun}}

\newcommand{\fil}[1]{(#1,\varphi_{#1})}

\newcommand{\C}{\mathcal{C}}
\newcommand{\D}{\mathcal{D}}

\newcommand{\N}{\mathbb{N}}
\newcommand{\Diag}{\text{Diag}}
\newcommand{\Colim}{\text{Colim}}

\newcommand{\RealNP}[1]{\Real{#1}_{N(P)}}
\newcommand{\Int}{\text{Int}}

\newcommand{\Strat}{\text{Strat}}

\newcommand{\StrStrat}{\text{StrStrat}}
\newcommand{\CP}{\mathcal{C}^0_P}
\newcommand{\CNP}{\mathcal{C}^0_{N(P)}}
\newcommand{\TopNP}{\Top_{N(P)}}
\renewcommand{\Im}{\text{Im}}

\title{Homotopy theory of stratified spaces}
\author{Sylvain Douteau}

\begin{document}
\begin{abstract}
In this article, we construct a cofibrantly generated model structure on the category of spaces stratified over a fixed poset, and show that it is Quillen-equivalent to a category of diagrams of simplicial sets. Then, considering all those model structures together, we construct a cofibrantly generated model structure on the category of all stratified spaces. In both model categories, weak-equivalences are characterized by stratified homotopy groups.
\end{abstract}
\maketitle

\tableofcontents
\section*{Introduction}

Stratified spaces are omnipresent in geometry and topology whenever dealing with singular objects. Stratifications were first defined by H. Whitney \cite{WhitneyStratification} for algebraic varieties, and R. Thom \cite{ThomEnsemblesMorphismesStratifies} for analytic varieties. Both defined a stratification as a decomposition of a variety $X$ into smooth pieces called strata, $X=\coprod_{p\in P}X_p$, together with some gluing conditions. 
The study of stratified spaces developed through the generalization of topological invariants of manifolds to the stratified setting. Among those invariants, intersection cohomology - introduced in 1980 by M. Goresky and R. MacPherson in \cite{IntersectionHomologyI} and \cite{IntersectionHomologyII} - is one of the first, and probably the most well-known. It is a "cohomology theory" suited for the study of stratified spaces and for which pseudo-manifolds - such as those defined by R. Thom and H. Whitney - satisfy Poincaré duality. Since then, numerous other invariants have been defined for stratifed spaces, such as 
\begin{itemize}
\item The category of exit-paths, and its higher-categorical generalizations, first defined by MacPherson (unpublished) and studied by Treuman \cite{Treumann}, J. Woolf \cite{WoolfFundamentalCategory} and J. Lurie \cite{HigherAlgebra}.
\item Singular chains and cochains complexes carrying cap and cup products and computing intersection (co)homology have been defined by H.King \cite{King} and G. Friedman and J. McClure \cite{FriedmanMcClure}, \cite{Friedman}.
\item A rational homotopy theory has been developed by D. Chataur, M. Saralegui and D. Tanré \cite{DavidMemoire}.
\item Generalized intersection homology theories have been proposed by M. Banagl \cite{BanaglRationalGeneralizedHomology}.
\end{itemize}
All those invariants are compatible with strata-preserving homotopy equivalences, though they are not preserved by regular homotopy equivalences. This begs the question - already asked by Goresky and MacPherson about intersection cohomology \cite[Problem \#4]{Borel}.
What is the corresponding homotopy category of stratified spaces that factor those invariants?

Several partial answers have been given toward the resolution of this problem.
D. Miller \cite{Miller}, gave a characterization of strata-preserving homotopy equivalences for the class of homotopically stratified spaces defined by Quinn \cite{Quinn}. Later, in his thesis, S. Nand-Lal \cite{NandLal} defined a partial model structure on the subcategory of "fibrant stratified spaces" using an adjunction with the category of simplicial sets. In another direction, both P.Haine \cite{Haine} and the author \cite{ArticleMoi} independently defined model structures on a category of filtered simplicial sets. All those results allowed to further the understanding of the homotopy theory of stratified spaces, but a description of a model structure on the entire category of stratified spaces was still missing.

The goal of this paper is to construct and characterize such a model structure. As is classical, we define a stratified space as a topological space $X$ together with a continuous map to a partially ordered set of strata $\varphi_X\colon X\to P$ (see for example \cite[Definition A.5.1]{HigherAlgebra}). We use the stratified homotopy groups, a stratified homotopy invariant introduced in \cite{ArticleMoi}, to characterize stratified weak-equivalences and prove the following (Theorem \ref{TheoremCMFStrat}).

\begin{theointro}\label{TheointroStrat}
The category of stratified spaces admits a cofibrantly generated model structure in which a map $f\colon (X\to P)\to (Y\to Q)$ is a weak-equivalence if and only if it induces isomorphisms on all stratified homotopy groups.
\begin{equation*}
s\pi_n(f)\colon s\pi_n((X\to P),\phi)\to s\pi_n((Y\to Q),f\circ\phi)\circ \widehat{f}
\end{equation*}
\end{theointro}

In order to prove Theorem \ref{TheointroStrat}, we first work over a fixed poset $P$, and construct a model structure on the category of spaces filtered over $P$, $\Top_P$ (Theorem \ref{TheoCMFTopP}).

\begin{theointro}\label{TheoIntroTopNP}
There exists a cofibrantly generated model structure on $\Top_P$ in which a map $f\colon (X\to P)\to (Y\to P)$ is a weak-equivalence if and only if it induces isomorphisms on all stratified homotopy groups.
\begin{equation*}
s\pi_n(f)\colon s\pi_n((X\to P),\phi)\to s\pi_n((Y\to P),f\circ\phi)
\end{equation*}
\end{theointro}
To construct this model structure, we consider an adjunction with a category of diagrams of simplicial sets
\begin{equation}\label{EquationIntroQE}
\Top_P\leftrightarrow \Fun(R(P)^{\op},\sS)
\end{equation}
where $R(P)$ can be succintly described as the subdivision of $P$.
The projective model structure on $\Fun(R(P)^{\op},\sS)$ can then be transported onto $\Top_P$ along the adjunction (\ref{EquationIntroQE}) by \cite[Corollary 3.3.4]{Hess}.
We then have the following characterization of $\Top_P$ (Theorems \ref{TheoCMFTopP} and \ref{TheoQETopNPDiagP}).
\begin{theointro}\label{TheoIntroQE}
The adjunction (\ref{EquationIntroQE}) is a Quillen-equivalence.
\end{theointro}
Using this Quillen-Equivalence, P. Haine \cite{Haine} was able to relate the aforementioned $\infty$-category of filtered spaces, $\Top_P$, and the $\infty$-category of décollages, $\text{Déc}_P$ considered by C. Barwick, S. Glassman and himself in \cite{Exodromy}.

In order to prove Theorem \ref{TheoIntroQE}, we introduce a more rigid notion of stratification. A strongly stratified space will be a space $X$ together with a strong stratification $\varphi_X\colon X\to \Real{N(P)}$. This allows one to consider the preimages of any point $t\in \Real{N(P)}$, $\varphi_X^{-1}(t)\subset X$, in addition to the strata of $X$. This idea was first considered by A. Henriques \cite{Henriques} - for the study of filtered simplicial sets through their realization - and proved to be extremely useful for the study of strongly stratified spaces.

\subsection*{Linear overview}

Section \ref{SectionStrat} covers the definitions needed in the rest of the paper.
In section \ref{SubsectionStrStrat}, we introduce strongly stratified spaces, strongly stratified map, and give some examples. The section also contains the notions of stratified homotopy equivalences and of various useful functors.
In section \ref{SubsectionSPiN}, we recall the definition of stratified homotopy groups from \cite{ArticleMoi}, although the definition is given in terms of strongly stratified spaces.
In section \ref{SubsectionStrat}, we make explicit the relationship between stratified and strongly stratified spaces.

Section \ref{SectionCMFTopP} contains the proof of Theorems \ref{TheoIntroTopNP} and \ref{TheoIntroQE}. It is concerned with the homotopy theory of spaces (strongly) filtered over a fixed poset $P$.
In section \ref{SubsectionDiagP}, we recall the projective model structure on a category of diagrams. The model structure on the category of strongly filtered spaces, $\Top_{N(P)}$ is constructed in section \ref{SubsectionCMFTopNP} (Theorem \ref{TheoremeCMFTopNP}), and its Quillen-equivalences with the category of diagrams $\Diag_P$ (Theorem \ref{TheoQETopNPDiagP}) and the category of filtered spaces $\Top_P$ (Theorem \ref{TheoCMFTopP}) are proved in section \ref{SubsectionQETopNPDiagP} and \ref{SubsectionCMFTopP} respectively.

Section \ref{SectionCMFStrat} contains the proof of Theorem \ref{TheointroStrat}.
In section \ref{SubsectionQAAlpha}, we prove that any map of posets $\alpha\colon P\to Q$ induces Quillen-adjunctions between $\Top_P$ and $\Top_Q$ which is a Quillen-equivalence if and only if $\alpha$ is an isomorphism. This is precisely hypothesis ($Q$) in \cite[Theorem 4.2]{Cagne}, which we apply in section \ref{SubsectionCMFStrat} to produce a model structure on $\Strat$ (Theorem \ref{TheoremCMFStrat}).

Appendix \ref{AppendixColimits} contains the proof of a simple but technical fact about particular colimits appearing in the proof of Theorem \ref{TheoQETopNPDiagP}.

\section{Strongly stratified spaces and stratified homotopy groups}
\label{SectionStrat}
In this section, we introduce the category of strongly stratified spaces, $\StrStrat$, and provide a few examples. We then recall a complete definition of the stratified homotopy groups from \cite{ArticleMoi} expressed in terms of strongly stratified spaces. Finally, after recalling the classical definition of a stratified space, we explain the relation between the categories of stratified spaces and of strongly stratified spaces.
\subsection{Strongly stratified, strongly filtered spaces and filtered simplicial sets}
\label{SubsectionStrStrat}
Throughout this paper, $\Top$ stands for the category of $\Delta$-generated spaces in the sense of \cite{Dugger}. In particular, all topological spaces are assumed to be $\Delta$-generated. This is needed in order for the category $\Top$ (and in turn the categories $\Top_P$ and $\Top_{N(P)}$) to be a locally presentable category.

\begin{defin}
A strongly stratified space is the data of a triple $(X,P,\varphi_X)$ where 
\begin{itemize}
\item $X$ is a topological space,
\item $P$ is a partially ordered set,
\item $\varphi_X\colon X\to \Real{N(P)}$ is a continuous map, where $\Real{N(P)}$ is the realisation of the nerve of $P$.
\end{itemize}
A (strongly) stratified map between strongly stratified spaces $f\colon (X,P,\varphi_X)\to (Y,Q,\varphi_Y)$ is the data of 
\begin{itemize}
\item a continuous map $f\colon X\to Y$,
\item a map of posets $\widehat{f}\colon P\to Q$,
\end{itemize}
such that the following diagram commutes
\begin{equation*}
\begin{tikzcd}[column sep=large]
X
\arrow[swap]{d}{\varphi_X}
\arrow{r}{f}
&Y
\arrow{d}{\varphi_Y}
\\
\Real{N(P)}
\arrow{r}{\Real{N(\widehat{f})}}
&\Real{N(Q)}
\end{tikzcd}
\end{equation*}
The category $\StrStrat$ is the category of strongly stratified spaces and (strongly) stratified maps.
\end{defin}

\begin{remarque}
We will drop the adverb "strongly" in the expression "strongly stratified map", since all the map we will be considering between strongly stratified spaces will be strongly stratified.
\end{remarque}

\begin{exemple}
\begin{itemize}
\item The most immediate class of example of strongly stratified spaces comes from triangulated spaces. Let $X\simeq \Real{K}$ be a triangulated space, and assume that some decomposition of $X$ into sub-triangulated spaces - the strata of X - is given $X=\coprod_{p\in P}X_p$, with $X_p=\Real{K_p}$, and $K_p\subset K$. This gives a map between set $\varphi_K\colon K_0\to P$ sending a vertex $v\in K$ to the $p\in P$ such that $v\in K_p$. If the strata of $X$ satisfy the frontier condition (we refer the reader to \cite[Section 3.1]{TamakiTanaka} for a definition of the frontier condition and a discussion of the various possible assumptions on a stratification), one can define a partial order on the set of strata of $X$ as
\begin{equation*}
p\leq q\Leftrightarrow X_p\cap \bar{X_q}\not = \emptyset.
\end{equation*}
The map of sets $\varphi_K$ then extends to a map of simplicial sets
\begin{equation*}
\varphi_K\colon K\to N(P)
\end{equation*}
Taking the realization of this map, one then has a strong stratification on $X$
\begin{equation*}
\varphi_X\colon X\simeq \Real{K}\to \Real{N(P)},
\end{equation*}
satisfying $\varphi_X^{-1}(\{p\})=X_p$ for all $p\in P$. See Definitions \ref{DefinitionFilteredSimplicialSet} and \ref{DefinitionStronglyFilteredRealization}.
\item Let $L$ be some topological space, and define its (closed) cone as
\begin{equation*}
\bar{c}(L)=L\times [0,1]/L\times\{0\}.
\end{equation*}
Identifying $\Real{N(\{p_0<p_1\})}$ with $[0,1]$ gives a strong stratification on $\bar{c}(L)$ :
\begin{equation*}
\varphi_{\bar{c}(L)}=\pr_{[0,1]}\colon \bar{c}(L)\simeq L\times[0,1]/L\times\{0\}\to [0,1]\simeq \Real{N(\{p_0<p_1\})}.
\end{equation*}
\item More generally, let $X$ be a pseudo-manifold with isolated singularities. That is, $X$ is a topological manifold outside of a discrete subset of points, and those points have neighborhoods homeomorphic to (closed) cones. Let $\{x_i\}_{i\in I}\subset X$ be the set of singular points, and fix some mutually disjoint neighborhood of the $x_i$, $F_i\subset X$, such that $F_i\simeq \bar{c}(L_i)$ for some manifold $L_i$, for $i\in I$.
Then, each of the $F_i$ admits a strong stratification to $\Real{N(\{p_i<q\})}$. Consider the poset $P=\{p_i\ | i\in I\}\cup \{q\}$, with relations $p_i<q$ for all $i\in I$. Taking all the strong stratifications $\varphi_{\bar{c}(L_i)}$ together - and defining $\varphi_X$ on the complement of the $F_i$ as $q$ - gives a strong stratification on $X$ : $\varphi_X\colon X\to \Real{N(P)}$.
\item Let $\varphi_L\colon L\to \Real{N(P)}$ be a strongly stratified space. Define the poset $c(P)$ as $P\coprod\{-\infty\}$, with $-\infty< p$ for all $p\in P$. Notice that $\Real{N(c(P))}$ can be identified with $c(\Real{N(P)})$. With this identification, $\bar{c}(L)$ admits a strong stratification to $\Real{N(c(P))}$.
\begin{equation*}
\varphi_{\bar{c}(L)}=\bar{c}(\varphi_L)\colon \bar{c}(L)\to \bar{c}\Real{N(P)}\simeq \Real{N(c(P))}.
\end{equation*}
\item Taking both previous examples together, one can define recursively a strong stratification on any pseudo-manifold given compatible trivializing neighborhoods for all strata.
\end{itemize}
\end{exemple}

\begin{defin}
Let $f,g\colon (X,P,\varphi_X)\to (Y,Q,\varphi_Y)$ be two stratified maps. We say that $f$ and $g$ are stratified homotopic if there exists a stratified homotopy $H\colon (X\times [0,1],P,\varphi_X\circ\pr_X)\to(Y,Q,\varphi_Y)$ such that $H_{|X\times\{0\}}=f$ and $H_{|X\times \{1\}}=g$. A stratified map $f\colon (X,P,\varphi_X)\to (Y,Q,\varphi_Y)$ is a stratified homotopy equivalence if there exists a stratified map $g\colon (Y,Q,\varphi_Y)\to (X,P,\varphi_X)$ such that $f\circ g$ and $g\circ f$ are stratified homotopic to $\Id_Y$ and $\Id_X$ respectively.
\end{defin}

\begin{remarque}\label{RemarqueFixedPoset}
From the definition of a stratified homotopy equivalence, we deduce that if $(X,P,\varphi_X)$ and $(Y,Q,\varphi_Y)$ are stratified homotopy equivalent, then $P$ and $Q$ must be isomorphic as posets. In particular, in order to understand the homotopy theory of stratified spaces, it is enough to work over a fixed poset. This is what we will call the \textbf{filtered} setting. This leads us to the following definition.
\end{remarque}

\begin{defin}
Let $P$ be a partially ordered set, the category $\Top_{N(P)}$ of strongly filtered spaces over $P$ is the subcategory of $\StrStrat$ whose objects are the strongly stratified spaces $(X,P,\varphi_X)$, and whose maps $f\colon (X,P,\varphi_X)\to (Y,P,\varphi_Y)$ satisfy $\widehat{f}=\Id_P$. When working with filtered objects, if the poset $P$ is clear from the context, we will drop it from the notation and write $\fil{X}$ instead of $(X,P,\varphi_X)$.
\end{defin}

\begin{defin}\label{DefinitionFilteredSimplicialSet}
Let $P$ be a partially ordered set. A filtered simplicial set over $P$ is the data of a pair $(K,\varphi_K)$, where
\begin{itemize}
\item $K$ is a simplicial set,
\item $\varphi_K\colon K\to N(P)$ is a simplicial map, where $N(P)$ is the nerve of $P$.
\end{itemize}
A filtered map between filtered simplicial sets, $f\colon (K,\varphi_K)\to (L,\varphi_L)$ is a simplicial map $f\colon K\to L$ such that the following diagram commutes
\begin{equation*}
\begin{tikzcd}
K
\arrow{rr}{f}
\arrow[swap]{dr}{\varphi_K}
&&L
\arrow{dl}{\varphi_L}
\\
&N(P)
\end{tikzcd}
\end{equation*}
The category $\sS_P$ is the category of filtered simplicial sets over $P$ and filtered maps.
\end{defin}

\begin{defin}\label{DefinitionStronglyFilteredRealization}
We define the strongly filtered realization functor $\RealNP{-}\colon \sS_P\to \Top_{N(P)}$ as follows. On objects, $\RealNP{(K,\varphi_K\colon K\to N(P))}=(\Real{K},\Real{\varphi_K})$, and on maps, $\RealNP{f\colon (K,\varphi_K)\to (L,\varphi_L)}=\Real{f}\colon \RealNP{(K,\varphi_K)}\to\RealNP{(L,\varphi_L)}$.
\end{defin}

\begin{prop}
The functor $\RealNP{-}\colon \sS_P\to \Top_{N(P)}$ admits a right adjoint, $\Sing_{N(P)}\colon \Top_{N(P)}\to \sS_P$ defined as the following pullback
\begin{equation*}
\begin{tikzcd}
\Sing_{N(P)}(X,\varphi_X)
\arrow{r}
\arrow{d}
&\Sing(X)
\arrow{d}{\Sing(\varphi_X)}
\\
N(P)
\arrow{r}
&\Sing(\Real{N(P)})
\end{tikzcd}
\end{equation*}
\end{prop}

\begin{proof}
See for example \cite[Definition A.6.2]{HigherAlgebra}
\end{proof}

\subsection{Stratified homotopy groups}
\label{SubsectionSPiN}
\begin{defin}\label{DefinitionDiagP}
Let $P$ be a poset. We define $R(P)$ as the full subcategory of $\sS_P$ whose objects are the non degenerate simplices of $N(P)$, $\varphi\colon \Delta^n\to N(P)$. We will write $\Delta^{\varphi}$, or $\{p_0<\dots<p_n\}$ for such a simplex, where $p_i$ is the image by $\varphi$ of the $i$-th vertex of $\Delta^n$.
The category $\Diag_P$ of simplicial diagrams over $P$ is the category of functors
\begin{equation*}
\Diag_P=\Fun(R(P)^{\op},\sS)
\end{equation*}
\end{defin}

\begin{defin}
Let $(X,\varphi_X)$ and $(Y,\varphi_Y)$ be two strongly filtered spaces. The inclusion 
\begin{equation*}
\Hom_{\Top_{N(P)}}(\fil{X},\fil{Y})\subset \mathcal{C}^0(X,Y),
\end{equation*}
 induces a topology on $\Hom_{\Top_{N(P)}}(\fil{X},\fil{Y})$. Write $\CP(\fil{X},\fil{Y})$ for the corresponding topological space. This defines a functor
\begin{equation*}
\C^0_{N(P)}\colon \TopNP^{\op}\times\TopNP\to \Top
\end{equation*}
\end{defin}

\begin{defin}
We define the functor $D\colon \Top_{N(P)}\to \Diag_P$ as follows
\begin{align*}
\Top_{N(P)}&\to \Diag_P\\
(X,\varphi_X)&\mapsto \Sing\left(\CNP\left(\RealNP{-},\fil{X}\right)\right)
\end{align*}
\end{defin}

\begin{defin}
Let $(X,\varphi_X)$ be a strongly filtered space over $P$. The stratified set of connected components of $(X,\varphi_X)$ is the functor $s\pi_0(X,\varphi_X)\colon R(P)^{\op}\to \Set$, obtained by composing $D$ with the functor $\pi_0\colon \sS\to\Set$.
\end{defin}

We want to define the stratified homotopy groups of a strongly stratified space $(X,P,\varphi_X)$, as the collection of homotopy groups of $D(\fil{X})(\Delta^{\varphi})$ for all $\Delta^{\varphi}$. In order to do so, one should first point each simplicial set $D(\fil{X})(\Delta^{\varphi})$ in a way that is compatible with the maps $D(\fil{X})(\Delta^{\varphi})\to D(\fil{X})(\Delta^{\psi})$ induced by the inclusions $\Delta^{\psi}\to\Delta^{\varphi}$. One way of obtaining such a collection of pointings is to consider a filtered map $\RealNP{N(P)}\to \fil{X}$. However, such a map need not always exist, even for non-pathological spaces, as illustrated by Example \ref{ExampleStephen}. Instead, we will consider partial collection of pointings, induced by strongly filtered maps $\RealNP{\Delta^{\varphi}}\to\fil{X}$ (Definition \ref{DefinitionStratifiedPointings}).

\begin{exemple}[{\cite[Example 10.1.0.10]{NandLal}}]\label{ExampleStephen}
Consider the stratified space of figure \ref{FigurePointings}, obtained as the realization of some stratified simplicial set. It is stratified over $P=\{a_1,a_2,a_3,b_1,b_2,b_3,c\}$ with relations $a_i<b_{i-1},b_{i+1}$, $1\leq i\leq 3$ where $i+1$ and $i-1$ are to be understood mod 3, and $b_j< c$ for $1\leq j\leq 3$. One checks that $s\pi_0(X,P,\varphi_X)$ is given by the functor
\begin{align*}
s\pi_0(X,P,\varphi_X)\colon R(P)^{\op}&\to \Set\\
\Delta^{\varphi}&\mapsto \{*\}
\end{align*}
And so, $(X,P,\varphi_X)$ is "connected" as a stratified space in the strongest possible sense, and it is the realization of a stratified simplicial set. Yet, there exists no stratified "pointing" of the form $\phi\colon\Real{N(P)}\to (X,P,\varphi_X)$.
\end{exemple}

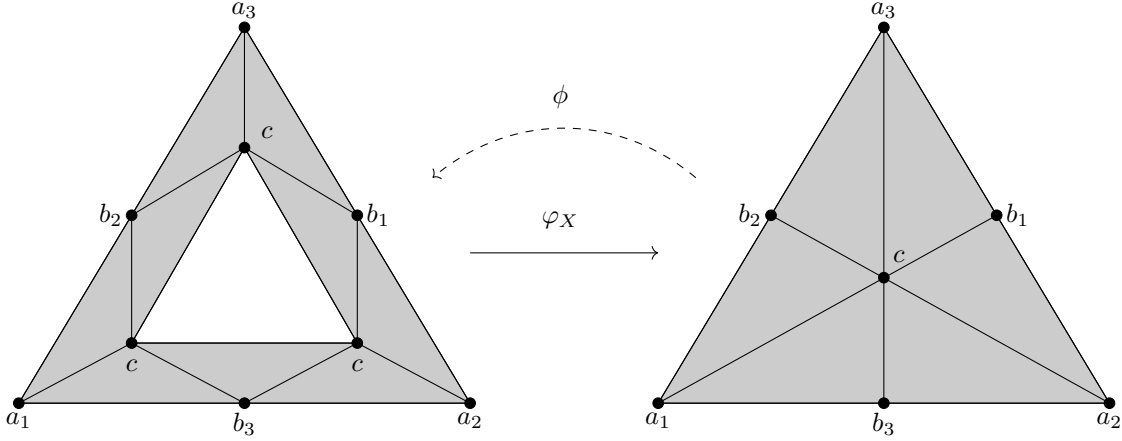
\begin{figure}[h]
\begin{tikzpicture}

\filldraw[fill=gray!40!white,draw = black](0,0)--(6,0)--(3,5)--(0,0);
\filldraw[fill=white, draw=black](1.5,0.8)--(4.5,0.8)--
(3,3.4)--(1.5,0.8);
\draw[black](0,0)--(6,0)--(3,5)--(0,0)--(1.5,0.8)--(1.5,2.5)--(3,3.4)--(3,5);
\draw[black](3,3.4)--(4.5,2.5)--(4.5,0.8)--(6,0);
\draw[black](4.5,0.8)--(3,0)--(1.5,0.8);
\filldraw[black] (0,0) circle (2pt);
\filldraw[black](6,0) circle (2pt);
\filldraw[black](3,5) circle (2pt);
\filldraw[black](4.5,2.5) circle (2pt);
\filldraw[black](1.5,2.5) circle (2pt);
\filldraw[black](3,0) circle (2pt);
\filldraw[black](1.5,0.8) circle (2pt);
\filldraw[black](4.5,0.8) circle (2pt);
\filldraw[black](3,3.4) circle (2pt);
\draw[black](1.5,0.8)--(4.5,0.8)--(3,3.4)--(1.5,0.8);
\coordinate[label = {[black]below : $a_1$}] (a1) at (0,0);
\coordinate[label = {[black]below : $a_2$}] (a1) at (6,0);
\coordinate[label = {[black]above : $a_3$}] (a1) at (3,5);
\coordinate[label = {[black]right : $b_1$}] (a1) at (4.5,2.5);
\coordinate[label = {[black]left : $b_2$}] (a1) at (1.5,2.5);
\coordinate[label = {[black]below : $b_3$}] (a1) at (3,0);
\coordinate[label = {[black]below : $c$}] (a1) at (1.5,0.7);
\coordinate[label = {[black]below : $c$}] (a1) at (4.5,0.7);
\coordinate[label = {[black]above : $c$}] (a1) at (3.3,3.4);

\filldraw[fill=gray!40!white,draw=black,cm ={1,0,0,1,(8.5cm,0cm)}](0,0)--(6,0)--(3,5)--(0,0);
\draw[->] (6,2)--(8.5,2);
\draw[->, dashed] (9,3)to [out= 140,in=40](5.5,3);
\coordinate[label ={[black]above : $\phi$}](a) at (7.2,3.8);
\coordinate[label ={[black]above : $\varphi_X$}](a) at (7.2,2.2);
\draw[black,cm ={1,0,0,1,(8.5cm,0cm)}](0,0)--(6,0)--(3,5)--(0,0);
\draw[black,cm ={1,0,0,1,(8.5cm,0cm)}](0,0)--(4.5,2.5);
\draw[black,cm ={1,0,0,1,(8.5cm,0cm)}](6,0)--(1.5,2.5);
\draw[black,cm ={1,0,0,1,(8.5cm,0cm)}](3,0)--(3,5);
\filldraw[black,cm ={1,0,0,1,(8.5cm,0cm)}] (0,0) circle (2pt);
\filldraw[black,cm ={1,0,0,1,(8.5cm,0cm)}](6,0) circle (2pt);
\filldraw[black,cm ={1,0,0,1,(8.5cm,0cm)}](3,5) circle (2pt);
\filldraw[black,cm ={1,0,0,1,(8.5cm,0cm)}](4.5,2.5) circle (2pt);
\filldraw[black,cm ={1,0,0,1,(8.5cm,0cm)}](1.5,2.5) circle (2pt);
\filldraw[black,cm ={1,0,0,1,(8.5cm,0cm)}](3,0) circle (2pt);
\filldraw[black,cm ={1,0,0,1,(8.5cm,0cm)}](3,1.67) circle (2pt);
\coordinate[cm ={1,0,0,1,(8.5cm,0cm)},label = {[black]below : $a_1$}] (a1) at (0,0);
\coordinate[cm ={1,0,0,1,(8.5cm,0cm)},label = {[black]below : $a_2$}] (a1) at (6,0);
\coordinate[cm ={1,0,0,1,(8.5cm,0cm)},label = {[black]above : $a_3$}] (a1) at (3,5);
\coordinate[cm ={1,0,0,1,(8.5cm,0cm)},label = {[black]right : $b_1$}] (a1) at (4.5,2.5);
\coordinate[cm ={1,0,0,1,(8.5cm,0cm)},label = {[black]left : $b_2$}] (a1) at (1.5,2.5);
\coordinate[cm ={1,0,0,1,(8.5cm,0cm)},label = {[black]below : $b_3$}] (a1) at (3,0);
\coordinate[cm ={1,0,0,1,(8.5cm,0cm)},label = {[black]above : $c$}] (a1) at (3.2,1.75);

\end{tikzpicture}
\caption{The space $X$ with a stratification over $\Real{N(P)}$, with $P=\{a_1,a_2,a_3,b_1,b_2,b_3,c\}$}
\label{FigurePointings}
\end{figure}

\begin{defin}\label{DefinitionStratifiedPointings}
Let $\fil{X}$ be a space strongly filtered over $P$. A filtered pointing of $\fil{X}$ is the data of a non-degenerate simplex of $N(P)$, $\Delta^{\varphi}$ and of a strongly filtered map
\begin{equation*}
\phi\colon \RealNP{\Delta^{\varphi}}\to \fil{X}.
\end{equation*} 
\end{defin}

\begin{remarque}
If $\varphi\colon \Delta^n\to N(P)$ is a non-degenerate simplex of $N(P)$, the vertices of $\Delta^n$ are sent by $\varphi$ to an increasing sequence of elements of $P$, $\varphi(e_0)<\varphi(e_1)<\dots<\varphi(e_n)$. We will write $\Im(\varphi)$ for this totally ordered subset of $P$. Note that $R(\Im(\varphi))$ is canonically equivalent to the full subcategory of $R(P)$ whose objects are the $\Delta^{\psi}\subseteq \Delta^{\varphi}$.
\end{remarque}

\begin{defin}
Let $\fil{X}$ be a space strongly filtered over $P$, and let $\phi\colon\RealNP{\Delta^{\varphi}}\to\fil{X}$ be a filtered pointing of $\fil{X}$. We define the pointed diagram of $\fil{X}$, 
\begin{equation*}
D_{\phi}\fil{X}\colon R(\Im(\varphi))^{op}\to \sS_*
\end{equation*}
 as the restriction of the functor $D(\fil{X})$ to the full subcategory $R(\Im(\varphi))^{\op}\subseteq R(P)^{\op}$, where each $D\fil{X}(\Delta^{\psi})$ is pointed by $\phi_{|\RealNP{\Delta^{\psi}}}$ for all $\Delta^{\psi}\subseteq \Delta^{\varphi}$. Here, $\sS_*$ is the category of pointed simplicial sets. Let $k\geq 1$, the $k$-th stratified homotopy group of $\fil{X}$ with respect to the pointing $\phi$ is the functor with values in the category of groups
\begin{equation*}
s\pi_k(\fil{X},\phi)\colon R(\Im(\varphi))^{\op}\to\text{Grp},
\end{equation*}
obtained by composing $D_{\phi}\fil{X}$ with the $k$-th homotopy group functor $\pi_k\colon \sS_*\to \text{Grp}$.
\end{defin}

\begin{remarque}
If $P=\{*\}$ is the poset with only one element, then any space $X$ is stratified over $P$. In this case, a stratified pointing is just a point $x\in X$, and the category $R(P)$ has a single object $\Delta^{*}$. For $n\geq 1$, one has $s\pi_n(X,x)(\Delta^{*})=\pi_n(X,x)$ and $s\pi_0(X)(\Delta^*)=\pi_0(X)$.
\end{remarque}

The following propositions follow easily from the functoriality of previous definitions.
\begin{prop}\label{PropositionSPiNFonctorial}
Let $f\colon (X,\varphi_X)\to(Y,\varphi_Y)$ be a filtered map. Then, for all $n\geq 0$ and all stratified pointings of $X$, $\phi\colon \RealNP{\Delta^{\varphi}}\to (X,\varphi_X)$, $f$ induces a natural transformation
\begin{equation*}
s\pi_n(f)\colon s\pi_n(\fil{X},\phi)\to s\pi_n(\fil{Y},f\circ\phi)
\end{equation*}
\end{prop}

\begin{prop}\label{PropositionSPiNHomotopyInvariant}
Let $\fil{X}$ be a space strongly filtered over $P$ and let $\phi,\phi'\colon \RealNP{\Delta^{\varphi}}\to\fil{X}$ be two pointings of $\fil{X}$. Then if $H\colon \RealNP{\Delta^{\varphi}}\times [0,1]\to\fil{X}$ is a stratified homotopy between $\phi$ and $\phi'$, then $H$ induces isomorphisms for all $n\geq 1$
\begin{equation*}
s\pi_n(H)\colon s\pi_n(\fil{X},\phi)\xrightarrow{\simeq} s\pi_n(\fil{X},\phi').
\end{equation*}
Furthermore, let $f,g\colon \fil{X}\to\fil{Y}$ be two filtered maps, if $H'\colon \fil{X}\times[0,1]\to\fil{Y}$ is a filtered homotopy between $f$ and $g$, then there is a commutative diagram
\begin{equation*}
\begin{tikzcd}
s\pi_n(\fil{X},\phi)
\arrow{r}{s\pi_n(f)}
\arrow[swap]{dr}{s\pi_n(g)}
&s\pi_n(\fil{Y},f\circ\phi)
\arrow{d}{s\pi_n(H'\circ\phi)}
\\
&s\pi_n(Y,g\circ \phi)
\arrow[phantom, shift left = 3]{u}{\simeq}
\end{tikzcd}
\end{equation*}
\end{prop}

\begin{remarque}
Although pointings of the form $\phi\colon\RealNP{N(P)}\to \fil{X}$ need not always exist (see Example \ref{ExampleStephen}), one might still want to consider them when they do. Given such a pointing, one can still define the stratified homotopy groups of $(X,\varphi_X)$ with respect to $\phi$, and in that case, one gets a functor from the full category $R(P)^{\op}$ to $\text{Grp}$. Furthermore, this functor is still invariant under stratified homotopy in the sense of Propositions \ref{PropositionSPiNFonctorial} and \ref{PropositionSPiNHomotopyInvariant}.
\end{remarque}

\subsection{Stratified and filtered spaces}
\label{SubsectionStrat}
\begin{defin}
A stratified space is the data of a triple $(X,P,\varphi_X)$ where
\begin{itemize}
\item $X$ is a topological space,
\item $P$ is a partially ordered set,
\item $\varphi_X\colon X\to P$ is a continuous map, where $P$ is given the Alexandrov topology.
\end{itemize}
A stratified map between stratified spaces $f\colon (X,P,\varphi_X)\to (Y,Q,\varphi_Y)$ is the data of 
\begin{itemize}
\item a continuous map $f\colon X\to Y$,
\item a map of posets $\widehat{f}\colon P\to Q$,
\end{itemize}
such that the following diagram commutes
\begin{equation*}
\begin{tikzcd}[column sep=large]
X
\arrow[swap]{d}{\varphi_X}
\arrow{r}{f}
&Y
\arrow{d}{\varphi_Y}
\\
P
\arrow{r}{\widehat{f}}
&Q
\end{tikzcd}
\end{equation*}
The category $\Strat$ is the category of stratified spaces and stratified maps. Let $P$ be some fixed poset, the category of filtered spaces over $P$, $\Top_P$ is the subcategory of $\Strat$ whose objects are of the form $(X,P,\varphi_X)$ and whose maps $f\colon (X,P,\varphi_X)\to (Y,P,\varphi_Y)$ verify $\widehat{f}=\Id_P$.
\end{defin}

Recall that a point in $\Real{N(P)}$ is represented by a pair, $(\{p_0<\dots<p_n\},t)$, where $\{p_0,\dots,p_n\}$ is a simplex of $N(P)$, and $t\in \Delta^n$ is a point in the standard $n$-simplex. If one adds the condition that $t$ lies in the interior of $\Delta^n$, then this representation is unique.
\begin{defin}\label{DefinitionPhiP}
Let $\varphi_P\colon \Real{N(P)}\to P$ be the following morphism
\begin{align*}
\varphi_P\colon \Real{N(P)}&\to P\\
\left(\{p_0<\dots <p_n\},t\in \Int(\Delta^n)\right)&\mapsto p_n
\end{align*}
We will write $\varphi_P\circ-\colon \Top_{N(P)}\to\Top_P$ for the functor sending $(X,P,\varphi_X)\in \Top_{N(P)}$ to $(X,P,\varphi_P\circ\varphi_X)$. Let $(Y,P,\varphi_Y)\in \Top_P$ be a $P$-filtered space. Consider the following pullback
\begin{equation*}
\begin{tikzcd}
Y\times_P \Real{N(P)}
\arrow{r}{\widetilde{\varphi_Y}}
\arrow{d}
&\Real{N(P)}
\arrow{d}{\varphi_P}
\\
Y
\arrow[swap]{r}{\varphi_Y}
&P
\end{tikzcd}
\end{equation*}
Define the functor $-\times_P\Real{N(P)}\colon \Top_P\to\Top_{N(P)}$ as the following functor 
\begin{align*}
-\times_P\Real{N(P)}\colon \Top_P&\to\Top_{N(P)}\\
(Y,P,\varphi_Y)&\mapsto(Y\times_P \Real{N(P)},P,\widetilde{\varphi_Y})
 \end{align*}
\end{defin}

\begin{prop}
The functor $-\times_P\Real{N(P)}$ is right adjoint to the functor $\varphi_P\circ-$
\end{prop}

\begin{remarque}
Thanks to the previous adjonction, one can define an adjunction between the category of filtered simplicial set and of filtered spaces as follows
\begin{equation*}
\RealP{-}=(\varphi_P\circ-)\RealNP{-}\colon \sS_P\leftrightarrow \Top_P\colon \Sing_{N(P)}(-\times_P\Real{N(P)})=\Sing_P
\end{equation*}
Similarly, one defines the stratified homotopy groups of a filtered space $(Y,P,\varphi_Y)$ as the filtered homotopy groups of $(Y\times_P\Real{N(P)},P,\widetilde{\varphi_Y})$. One checks that this definition coincides with the definition given in \cite{ArticleMoi}
\end{remarque}

\section{The model category of filtered spaces}
\label{SectionCMFTopP}
We saw in Remark \ref{RemarqueFixedPoset} that it was sufficient to work over a fixed poset to understand the homotopy theory of stratified spaces. In this section, some poset $P$ is fixed, and we construct and characterize model structures for the categories of filtered and strongly filtered spaces over $P$ - the construction of the model structure on $\Strat$ is the content of section \ref{SectionCMFStrat}. We first recall the projective model structure on the category of diagrams of simplicial sets, $\Diag_P$ (see Definition \ref{DefinitionDiagP}) and then transport this model structure on $\Top_{N(P)}$ (Theorem \ref{TheoremeCMFTopNP}). We then go on to show that the two model structure are Quillen-equivalent (Theorem \ref{TheoQETopNPDiagP}), and that one can also transport a model structure to $\Top_P$ that is Quillen-equivalent to the one on $\Top_{N(P)}$ (Theorem \ref{TheoCMFTopP}).

\subsection{The category of diagrams}
\label{SubsectionDiagP}
Categories of functors between a small category and a cofibrantly generated model category always admit cofibrantly generated model structures. We recall how to define the class of generating (trivial) cofibrations in the case of $\Diag_P$. We also exhibit an adjunction $\Diag_P\leftrightarrow\Top_{N(P)}$.
\begin{defin}
Let $\Delta^{\varphi}\in R(P)$ be a non-degenerate simplex and $K$ a simplicial set. Define $K^{\Delta^{\varphi}}\colon R(P)^{\op}\to \sS$ as the following functor
\begin{align*}
K^{\Delta^{\varphi}}\colon R(P)^{\op}&\to \sS\\
\Delta^{\psi}&\mapsto \left\{\begin{array}{cl}
K &\text{ if $\Delta^{\psi}\subseteq\Delta^{\varphi}$}\\
\emptyset & \text{ else}
\end{array}\right.
\end{align*}
Where, for $\Delta^{\psi_1}\subseteq\Delta^{\psi_2}\subseteq\Delta^{\varphi}$, the maps $K^{\Delta^{\varphi}}(\Delta^{\psi_1}\to\Delta^{\psi_2})$ are defined as $\Id_K\colon K\to K$.
\end{defin}
\begin{prop}
The category $\Diag_P$ together with the following classes of maps is a cofibrantly generated model category.
\begin{itemize}
\item Fibrations are the maps $f\colon F\to G$ such that for all $\Delta^{\varphi}\in R(P)$, $f(\Delta^{\varphi})\colon F(\Delta^{\varphi})\to G(\Delta^{\varphi})$ is a Kan fibration.
\item Weak-equivalences are the maps $f\colon F\to G$ such that for all $\Delta^{\varphi}\in R(P)$, $f(\Delta^{\varphi})\colon F(\Delta^{\varphi})\to G(\Delta^{\varphi})$ is a weak-equivalence in the Kan model structure on $\sS$.
\item Cofibrations are the maps with the left lifting property against trivial fibrations.
\end{itemize}
The sets of generating cofibrations and trivial cofibrations are the following
\begin{itemize}
\item $
I=\left\{(\partial(\Delta^n)\to\Delta^n)^{\Delta^{\varphi}}\ |\ 0\leq n,\ \Delta^{\varphi}\in R(P)\right\}
$
\item $
J=\left\{(\Lambda^n_k\to \Delta^n)^{\Delta^{\varphi}}\ |\ 0\leq k\leq n, \ n>0,\ \Delta^{\varphi}\in R(P)\right\}
$
\end{itemize}
\end{prop}
\begin{proof}
This is simply the projective model structure on $\Fun(R(P)^{\op},\sS)$. See \cite[Theorem 11.6.1]{Hirschhorn}.
\end{proof}

\begin{defin}
We define the bi-functor $-\otimes-\colon \sS\times \sS_P\to \sS_P$ as follows : 
\begin{align*}
-\otimes-\colon \sS\times \sS_P&\to \sS_P\\
K\otimes \fil{X}&\mapsto (K\times X,\varphi_X\circ \pr_X)\\
(f\colon K\to L)\otimes (g\colon \fil{X}\to\fil{Y})&\mapsto f\times g\colon K\otimes\fil{X}\to L\otimes\fil{Y}
\end{align*}
\end{defin}

\begin{remarque}
We will also write $-\otimes-$ for the bi-functor $\Top\times\Top_{N(P)}\to \Top_{N(P)}$ defined in a similar way. The conflict of notation will not be problematic because, thanks to the classical result of J. Milnor \cite{MilnorGeometricRealization}, one has $\RealNP{K\otimes \fil{X}}\simeq \Real{K}\otimes\RealNP{\fil{X}},$. if either $K$ or $X$ has only finitely many non-degenerate simplices.
\end{remarque}

\begin{defin}\label{DefinitionColim}
Let $\mathcal{C}\subset R(P)^{\op}\times R(P)$ be the full subcategory whose objects are of the form $(\Delta^{\varphi},\Delta^{\psi})$ with $\Delta^{\psi}\subseteq \Delta^{\varphi}$.
Define the functor $-\otimes R(P)$ as follows
\begin{align*}
-\otimes R(P)\colon \Diag_P&\to \Fun(\mathcal{C},\sS_P)\\
F&\mapsto \left\{\begin{array}{ccc}
\mathcal{C}&\to & \sS_P\\
(\Delta^{\varphi},\Delta^{\psi})&\mapsto &F(\Delta^{\varphi})\otimes \Delta^{\psi}
\end{array}\right.
\end{align*}
Define the functor $\Colim\colon \Diag_P\to \Top_{N(P)}$ as follows \begin{align*}
\Colim\colon \Diag_P&\to \Top_{N(P)}\\
F&\mapsto \colim_{\C}\RealNP{F\otimes R(P)}
\end{align*}
\end{defin}

\begin{prop}
The functor $\Colim\colon \Diag_P\to \Top_{N(P)}$ is left adjoint to the functor $D\colon \Top_{N(P)}\to\Diag_P$.
\end{prop}  
\begin{proof}
Let $F\in \Diag_P$ be a functor and $\fil{X}$ be a strongly stratified space over $P$. A map $f\colon F\to D\fil{X}$ is the coherent data of a map of simplicial sets \begin{equation*}
f_{\Delta^{\varphi}}\colon F(\Delta^{\varphi})\to \Sing\left(\C^0_{N(P)}\left(\RealNP{\Delta^{\varphi}},\fil{X}\right)\right),
\end{equation*}
 for all $\Delta^{\varphi}\in R(P)$. By the adjunction $(\Real{-},\Sing)$, this is equivalent to the coherent data of maps
\begin{equation*}
\Real{F(\Delta^{\varphi})}\to \C^0_{N(P)}\left(\RealNP{\Delta^{\varphi}},\fil{X}\right),
\end{equation*} 
By Lemma \ref{LemmeAdjunctionRealC}, and using the fact that $\Real{F(\Delta^{\varphi})}\otimes \RealNP{\Delta^{\varphi}}\simeq \RealNP{F(\Delta^{\varphi})\otimes\Delta^{\varphi}}$ this is equivalent to the coherent data of maps
\begin{equation*}
\RealNP{F(\Delta^{\varphi})\otimes\Delta^{\varphi}}\to \fil{X}
\end{equation*} 
Which corresponds to a map $\Colim(F)\to \fil{X}$.
\end{proof}

\begin{lemme}\label{LemmeAdjunctionRealC}
For all $\Delta^{\varphi}\in R(P)$, the functor $-\otimes\RealNP{\Delta^{\varphi}}$ is left adjoint to the functor $\C^0_{N(P)}(\RealNP{\Delta^{\varphi}},-)$.
\end{lemme}

\begin{proof}
This follows from the definition of $\C^0_{N(P)}$ and the fact that $\Top$ is cartesian closed (recall that $\Top$ stands for the category of $\Delta$-generated space, see \cite{Dugger}).
\end{proof}

\subsection{The model category $\Top_{N(P)}$}
\label{SubsectionCMFTopNP}
In this section, we prove the following theorem.
\begin{theo}\label{TheoremeCMFTopNP}
There exists a cofibrantly generated model structure on $\Top_{N(P)}$ where a map $f\colon \fil{X}\to\fil{Y}$ is
\begin{itemize}
\item a fibration if $D(f)$ is a fibration,
\item a weak-equivalence if $D(f)$ is a weak-equivalence,
\item a cofibration if it has the left lifting property against all trivial fibrations.
\end{itemize}
The set of generating cofibrations and trivial cofibrations are
\begin{itemize}
\item $
I=\left\{\RealNP{\partial(\Delta^n)\otimes \Delta^{\varphi}\to \Delta^n\otimes\Delta^{\varphi}}\ |\ n\geq 0,\ \Delta^{\varphi}\in R(P)\right\}
$
\item $
J=\left\{\RealNP{\Lambda^n_k\otimes \Delta^{\varphi}\to \Delta^n\otimes\Delta^{\varphi}}\ |\ 0\leq k\leq n\ n> 0,\ \Delta^{\varphi}\in R(P)\right\}
$
\end{itemize}
\end{theo}

\begin{remarque}\label{RemarqueWeakEquivalenceIsoGroupesHomtopies}
A weak-equivalence in $\Top_{N(P)}$ is a map, $f\colon\fil{X}\to\fil{Y}$ inducing weak-equivalences $D(f)(\Delta^{\varphi})\colon D\fil{X}(\Delta^{\varphi})\to D\fil{Y}(\Delta^{\varphi})$ for all $\Delta^{\varphi}\in N(P)$. Equivalently, $f$ is a weak-equivalence, if and only if it induces isomorphisms of functors $s\pi_n(f)\colon s\pi_n(\fil{X},\phi)\to s\pi_n(\fil{Y},f\circ\phi)$, for all pointings of $X$ and for all $n$.
\end{remarque}

\begin{proof}
By lemma \ref{LemmeHypotheseCategoriques}, \cite[Corollary 3.3.4]{Hess} applies and  we only need to show that any filtered map having the left lifting property against any fibration is a weak-equivalence (in the sense of Theorem \ref{TheoremeCMFTopNP}). Let $f\colon\fil{X}\to\fil{Y}$ be such a filtered map. By lemma \ref{LemmeFactorisationRetract}, $f$ admits a factorization of the form 
\begin{equation*}
\begin{tikzcd}
\fil{X}
\arrow[swap]{r}{i}
&\fil{Z}
\arrow[swap,bend right,shift right = 1]{l}{r}
\arrow{r}{q}
&\fil{Y}
\end{tikzcd}
\end{equation*}
Consider the following lifting problem
\begin{equation*}
\begin{tikzcd}
\fil{X}
\arrow{r}{i}
\arrow[swap]{d}{f}
&\fil{Z}
\arrow{d}{q}
\\
\fil{Y}
\arrow[swap]{r}{\Id_Y}
\arrow[dashrightarrow]{ur}{g}
&\fil{Y}
\end{tikzcd}
\end{equation*}
By lemma \ref{LemmeFactorisationRetract}, $q$ is a fibration, hence there exists a lift $g\colon \fil{Y}\to\fil{Z}$. Fixing some pointing of $X$, $\phi$ and passing to stratified homotopy groups, we get
\begin{equation*}
\begin{tikzcd}
s\pi_n(\fil{X},\phi)
\arrow{r}{s\pi_n(i)}
\arrow[swap]{d}{s\pi_n(f)}
&s\pi_n(\fil{Z},i\circ\phi)
\arrow{d}{s\pi_n(q)}
\\
s\pi_n(\fil{Y},f\circ\phi)
\arrow[swap]{r}{\Id}
\arrow{ur}{s\pi_n(g)}
&s\pi_n(\fil{Y},f\circ\phi)
\end{tikzcd}
\end{equation*}
Since $\Id$ and $s\pi_n(i)$ are isomorphisms (By lemma \ref{LemmeFactorisationRetract}), so is $s\pi_n(g)$ since it admits left and right inverses. But then, $s\pi_n(f)$ is an isomorphism with inverse $\left(s\pi_n(i)\right)^{-1}s\pi_n(g)$. By Remarque \ref{RemarqueWeakEquivalenceIsoGroupesHomtopies}, this implies that $f$ is a weak-equivalence, which concludes the proof of the first part of theorem \ref{TheoremeCMFTopNP}. The set of generating cofibrations and trivial cofibrations are then obtained by applying the functor $\Colim$ to the corresponding generating set in $\Diag_P$.
\end{proof}

\begin{lemme}\label{LemmeHypotheseCategoriques}
The model category $\Diag_P$ is accessible, and the category $Top_{N(P)}$ is locally presentable.
\end{lemme}

\begin{proof}
The model category $Diag_P$ is cofibrantly generated, and locally presentable so it is accessible (see \cite{Hess}). In this text, $\Top$ is the category of $\Delta$-generated spaces, which is locally presentable (see \cite{Dugger}). Since $\Top_{N(P)}=\Top/\Real{N(P)}$, it is also locally presentable. 
\end{proof}

\begin{lemme}\label{LemmeFactorisationRetract}
Let $f\colon \fil{X}\to\fil{Y}$ be a filtered map. There exists a strongly filtered space, $\fil{Z}$, a fibration $q\colon \fil{Z}\to\fil{Y}$ , a filtered map $i\colon \fil{X}\to\fil{Z}$ and a filtered map $r\colon \fil{Z}\to\fil{X}$ such that 
\begin{itemize}
\item $f=q\circ i$,
\item $r\circ i=\Id_X$ and $i\circ r$ is filtered homotopic to $\Id_Z$.
\end{itemize}
\end{lemme}

\begin{proof}
Let $f\colon\fil{X}\to\fil{Y}$ be a filtered map. Define the topological space $Z$ as the following subspace of $X\times Y^{[0,1]}$ :
\begin{equation*}
Z=\{(x,\gamma)\in X\times Y^{[0,1]}\ |\ f(x)=\gamma(0),\ \varphi_Y(\gamma(t))=\varphi_X(x),\ \forall t\in [0,1]\},
\end{equation*}
and $\varphi_Z\colon Z\to \Real{N(P)}$ as the following composition
\begin{equation*}
\begin{tikzcd}
Z
\arrow[hookrightarrow]{r}
&X\times Y^{[0,1]}
\arrow{r}{\pr_X}
&X
\arrow{r}{\varphi_X}
&\Real{N(P)}
\end{tikzcd}
\end{equation*}
Define $i,q$ and $r$ as follows
\begin{align*}
i\colon \fil{X}&\to\fil{Z}\\
x&\mapsto (x,t\mapsto f(x))\\
\phantom{X}&\phantom{X}\\
q\colon \fil{Z}&\to\fil{Y}\\
(x,\gamma)&\mapsto \gamma(1)\\
\phantom{X}&\phantom{X}\\
r\colon \fil{Z}&\to\fil{X}\\
(x,\gamma)&\mapsto x
\end{align*}
One has $f=q\circ i$ and $r\circ i=\Id_X$, and the filtered map
\begin{align*}
(Z\times [0,1],\varphi_Z\circ\pr_Z)&\to \fil{Z}\\
((x,\gamma),s)&\mapsto (x,(t\mapsto \gamma(st))
\end{align*}
is a filtered homotopy between $i\circ r$ and $\Id_Z$. The only thing left to check is that $q$ is a fibration. Consider the following lifting problem
\begin{equation}\label{DiagrammeCommutatifPreuveTopNP}
\begin{tikzcd}
\Real{\Lambda^n_k}\otimes\RealNP{\Delta^{\varphi}}
\arrow{r}{\alpha}
\arrow[swap]{d}{j}
&\fil{Z}
\arrow{d}{q}
\\
\Real{\Delta^n}\otimes\RealNP{\Delta^{\varphi}}
\arrow{r}{\beta}
\arrow[dashrightarrow]{ur}{h}
&\fil{Y}
\end{tikzcd}
\end{equation}
The inclusion $\Real{\Lambda^n_k}\to\Real{\Delta^n}$ admits a deformation retract $\Real{\Delta^n}\to\Real{\Lambda^n_k}$. Taking the product with $\Id_{\Delta^{\varphi}}$, one gets a filtered retract of $j$, $s\colon \Real{\Delta^n}\otimes\RealNP{\Delta^{\varphi}}\to \Real{\Lambda^n_k}\otimes\RealNP{\Delta^{\varphi}}$ together with a filtered homotopy between $j\circ s$ and $\Id_{\Real{\Delta^n}\otimes\RealNP{\Delta^{\varphi}}}$.
\begin{equation*}
H\colon [0,1]\otimes \Real{\Delta^n}\otimes\RealNP{\Delta^{\varphi}}\to \Real{\Delta^n}\otimes\RealNP{\Delta^{\varphi}}.
\end{equation*}
Fix some continuous map $d\colon \Real{\Delta^n}\to \mathbb{R}_{\geq 0}$ such that $h^{-1}(0)=\Real{\Lambda^n_k}$ (for example, the distance to $\Real{\Lambda^n_k}$), and for $a\in \Real{\Lambda^n_k}\otimes\RealNP{\Delta^{\varphi}}$, write $\alpha(a)=(\alpha_X(a),\alpha_Y(a))\in Z\subset X\times Y^{[0,1]}$. With those notations, we will define $h$ as $(h_X,h_Y)$.
Define $h_X$ as $\alpha_X\circ s$, and $h_Y$ as follows
\begin{align*}
h_Y\colon \Real{\Delta^n}\otimes\RealNP{\Delta^{\varphi}}&\to Y^{[0,1]}\\
a&\mapsto \left\{\begin{array}{ccc}
[0,1]&\to & Y\\
t&\mapsto & \left\{\begin{array}{ll}
\alpha_Y(s(a))(t(1+d(a))) &\text{if $0\leq t\leq \frac{1}{1+d(a)}$}\\
\beta(H\left(a,\frac{1+d(a)}{d(a)}(t-\frac{1}{1+d(a)})\right)&\text{if $\frac{1}{1+d(a)}<t\leq 1$}
\end{array}\right.
\end{array}\right.
\end{align*}
One checks that $h=(h_X,h_Y)$ is indeed a lift for the problem (\ref{DiagrammeCommutatifPreuveTopNP}).
\end{proof}

\subsection{Quillen equivalence between $\Top_{N(P)}$ and $\Diag_P$}
\label{SubsectionQETopNPDiagP}
In this section, we prove the following theorem.
\begin{theo}\label{TheoQETopNPDiagP}
The adjunction 
\begin{equation*}
\Colim \colon \Diag_P\leftrightarrow Top_{N(P)}\colon D
\end{equation*}
is a Quillen equivalence.
\end{theo}

\begin{proof}
By definition of the model structure on $\Top_{N(P)}$, $D$ preserves fibrations and weak-equivalences, hence $(\Colim,D)$ is a Quillen adjunction. To show that it is a Quillen equivalence, one needs to prove that for any cofibrant object $F$ in $\Diag_P$, the unit $\eta_F\colon F\to D(\Colim(F))$ is a weak-equivalence. By definition of the model structure on $\Diag_P$, this amount to showing that for all $\Delta^{\varphi}\in R(P)$, the map
\begin{equation*}
F(\Delta^{\varphi})\to\Sing\left(\C^0_{N(P)}\left(\RealNP{\Delta^{\varphi}},\Colim(F)\right)\right)
\end{equation*}
is a weak-equivalence of simplicial sets. By adjunction, this is equivalent to showing that the map
\begin{equation*}
\Real{F(\Delta^{\varphi})}\to \C^0_{N(P)}\left(\RealNP{\Delta^{\varphi}},\Colim(F)\right)
\end{equation*}
is a weak-equivalence of topological spaces. We will factor this map as follows 
\begin{equation*}
\begin{tikzcd}
\Real{F(\Delta^{\varphi})}
\arrow{r}{f}
&\C^0_{N(P)}\left(\RealNP{\Delta^{\varphi}},\Real{F(\Delta^{\varphi})}\otimes\RealNP{\Delta^{\varphi}}\right)
\arrow{r}{g}
&\C^0_{N(P)}\left(\RealNP{\Delta^{\varphi}},\Colim(F)\right)
\end{tikzcd}
\end{equation*}
Where $f$ is the obvious map sending $x\in \Real{F(\Delta^{\varphi})}$ to $t\mapsto (x,t)$, and $g$ comes from the map \begin{equation*}
i_{\Delta^{\varphi}}\colon\Real{F(\Delta^{\varphi})}\otimes\RealNP{\Delta^{\varphi}}\to \colim \Real{F(\Delta^{\varphi})}\otimes\RealNP{\Delta^{\psi}}=\Colim(F)
\end{equation*}
The map $f$ admits a retraction, $r$, given by picking a point $t_0\in \RealNP{\Delta^{\varphi}}$, evaluating at $t_0$ then projecting the result onto $\Real{F(\Delta^{\varphi})}$. Choosing some deformation retract from $\Real{\Delta^n}$ to $\{t_0\}$, where $\Delta^n$ is the non-stratified simplex underlying $\Delta^{\varphi}$, then gives a homotopy between $f\circ r$ and $\Id$. In particular, $f$ is a homotopy equivalence. We will show that $g$ is a homeomorphism, which will conclude the proof. 
First, notice that 
\begin{equation*}
\Colim(F)=\bigcup_{\Delta^{\varphi}\in R(P)} \Im(i_{\Delta^{\varphi}})= \bigcup_{\Delta^{\varphi}\in R(P)} i_{\Delta^{\varphi}}\left(\Real{F(\Delta^{\varphi})}\otimes\Int\left(\RealNP{\Delta^{\varphi}}\right)\right),
\end{equation*}
where $\Int\left(\RealNP{\Delta^{\varphi}}\right)$ means the set of points of $\RealNP{\Delta^{\varphi}}$ which are not in the boundary $\RealNP{\partial(\Delta^{\varphi})}$.
In particular, this implies that for all $\Delta^{\varphi}\in R(P)$, the following diagram is cartesian
\begin{equation}\label{EquationCarreCartesienColim}
\begin{tikzcd}
i_{\Delta^{\varphi}}\left(\Real{F(\Delta^{\varphi})}\otimes\Int\left(\RealNP{\Delta^{\varphi}}\right)\right)
\arrow[hookrightarrow]{r}
\arrow{d}
&\Colim(F)
\arrow{d}
\\
\Int\left(\RealNP{\Delta^{\varphi}}\right)
\arrow[hookrightarrow]{r}
&\Real{N(P)}
\end{tikzcd}
\end{equation}
Let $\phi\colon \RealNP{\Delta^{\varphi}}\to \Colim(F)$ be a filtered map. Write $\phi_{\Int}$ its restriction to $\Int\left(\RealNP{\Delta^{\varphi}}\right)$. Since the square (\ref{EquationCarreCartesienColim}) is cartesian, the map $\phi_{\Int}$ factors through $i_{\Delta^{\varphi}}\left(\Real{F(\Delta^{\varphi}}\otimes\Int\left(\RealNP{\Delta^{\varphi}}\right)\right)$. 
In particular, $\phi$ must take value in the closure
\begin{equation*}
\overline{i_{\Delta^{\varphi}}\left(\Real{F(\Delta^{\varphi})}\otimes\Int\left(\RealNP{\Delta^{\varphi}}\right)\right)}\subset \Colim(F)
\end{equation*}
By lemma \ref{LemmeCofibrantDiagramHomeoOntoImage} , this closure is precisely the image of $\Real{F(\Delta^{\varphi})}\otimes\RealNP{\Delta^{\varphi}}$ in $\Colim(F)$, which is homeomorphic to $\Real{F(\Delta^{\varphi})}\otimes\RealNP{\Delta^{\varphi}}$. In conclusion, all filtered maps $\phi\colon\Delta^{\varphi}\to \Colim(F)$ factor through the subspace $\Real{F(\Delta^{\varphi})}\otimes\RealNP{\Delta^{\varphi}}\subset \Colim(F)$, and so the inclusion $g$ is in fact a homeomorphism.
\end{proof}

\begin{lemme}\label{LemmeCofibrantDiagramHomeoOntoImage}
Let $F$ be a cofibrant object of $\Diag_P$. Then, for all $\Delta^{\varphi}\in R(P)$, the map
\begin{equation*}
i_{\Delta^{\varphi}}\colon \Real{F(\Delta^{\varphi})}\otimes\RealNP{\Delta^{\varphi}}\to \Colim(F)
\end{equation*}
is a homeomorphism onto its image, and its image is closed in $\Colim(F)$.
\end{lemme}

\begin{proof}
Given a cofibrant diagram $F$, the space underlying $\Colim(F)$ is the realisation of the simplicial set $\colim\limits_{\C}F\otimes R(P)$. In particular, it is enough to show that the map of simplicial sets $F(\Delta^{\varphi})\otimes\Delta^{\varphi}\to \colim\limits_{\C}F\otimes R(P)$ is a monomorphism. First notice that if a diagram is cofibrant, then for any inclusion $\Delta^{\varphi}\subset \Delta^{\psi}$, the map $F(\Delta^{\psi})\to F(\Delta^{\varphi})$ is a monomorphism. Indeed, given any trivial Kan fibration $q\colon K\to L$, one proves that $F(\Delta^{\psi})\to F(\Delta^{\varphi})$ admits the left lifting property against $q$ by considering the lifting property of $\emptyset\to F$ with respect to the trivial fibration $G\to H$ where $G$ and $H$ are the diagrams respectively defined as :
\begin{equation*}
G(\Delta^{\mu})=\left\{\begin{array}{cl}
L&\text{ if $\Delta^{\mu}\subsetneq \Delta^{\varphi}$}\\
K&\text { else}
\end{array}\right.
 \text{ and }
 H(\Delta^{\mu})=\left\{\begin{array}{cl}
L&\text{ if $\Delta^{\mu}\subseteq \Delta^{\varphi}$}\\
K&\text { else}
\end{array}\right. 
\end{equation*}
We can then conclude by using Remark \ref{RemarqueSeenAsSet} and Lemma \ref{LemmeCAlmostFiltered} to apply Proposition \ref{PropositionColimiteMono}.
\end{proof}

\subsection{Quillen equivalence between the categories of filtered and strongly filtered spaces}
The goal of this section is to prove Theorem \ref{TheoCMFTopP} which, together with Theorem \ref{TheoQETopNPDiagP} provides a proof of Theorem \ref{TheoIntroQE}.
\label{SubsectionCMFTopP}
\begin{defin}
Let $D_P\colon \Top_P\to\Diag_P$ be the composition of $D$ with $-\times_P\Real{N(P)}$, where $-\times_P\Real{N(P)}$ is the functor from Definition \ref{DefinitionPhiP}.
\end{defin}

\begin{theo}\label{TheoCMFTopP}
There exists a cofibrantly generated model structure on $Top_{P}$ where a map $f\colon \fil{X}\to\fil{Y}$ is 
\begin{itemize}
\item a fibration if $D_P(f)$ is a fibration,
\item a weak-equivalence if $D_P(f)$ is a weak-equivalence,
\item a cofibration if it has the left lifting property against all trivial fibrations.
\end{itemize}
The set of generating cofibrations and trivial cofibrations are
\begin{itemize}
\item $
I=\left\{\RealP{\partial(\Delta^n)\otimes \Delta^{\varphi}\to \Delta^n\otimes\Delta^{\varphi}}\ |\ n\geq 0,\ \Delta^{\varphi}\in R(P)\right\}
$
\item $
J=\left\{\RealP{\Lambda^n_k\otimes \Delta^{\varphi}\to \Delta^n\otimes\Delta^{\varphi}}\ |\ 0\leq k\leq n\ n> 0,\ \Delta^{\varphi}\in R(P)\right\}
$
\end{itemize}
Furthermore, the adjunction
\begin{equation*}
\varphi_P\circ-\colon \Top_{N(P)}\leftrightarrow \Top_P\colon -\times_P\Real{N(P)}
\end{equation*}
is a Quillen-equivalence.
\end{theo}

\begin{proof}
The proof of Theorem \ref{TheoremeCMFTopNP} generalizes directly to show the existence of the model structure.
Let $f\colon\fil{X}\to\fil{Y}$ be a morphism of $\Top_P$. It is a (trivial) fibration if and only if $D_P(f)\colon D_P\fil{X}\to D_P\fil{Y}$ is a (trivial) fibration. The latter is true if and only if $f\times_P\Real{N(P)}\colon X\times_P\Real{N(P)}\to Y\times_P\Real{N(P)}$ is a (trivial) fibration of $\Top_{N(P)}$. In particular, $(\varphi_P\circ -,-\times_P\Real{N(P)})$ is a Quillen-adjunction. Let $\fil{X}\in \Top_{N(P)}$ be a strongly filtered space, $\fil{Y}\in\Top_P$ be a filtered space, and $f\colon (X,\varphi_P\circ\varphi_X)\to\fil{Y}$ be a filtered map. We need to show that $f$ is a weak-equivalence of $\Top_P$ if and only if its adjoint map is a weak-equivalence\,:
\begin{equation*}
\widetilde{f}\colon \fil{X}\to \fil{Y}\times_P\Real{N(P)}.
\end{equation*}
Notice that $\widetilde{f}$ factors as
\begin{equation*}
\begin{tikzcd}[column sep = huge]
\fil{X}
\arrow{r}{\epsilon_X}
&(X,\varphi_P\circ\varphi_X)\times_P\Real{N(P)}
\arrow{r}{f\times_P\Real{N(P)}}
&\fil{Y}\times_P\Real{N(P)}
\end{tikzcd}
\end{equation*}
By definition $f\times_P\Real{N(P)}$ is a weak-equivalence of $\Top_{N(P)}$ if and only if $D(f\times_P\Real{N(P)})$ is a weak-equivalence of $\Diag_P$ if and only if $f$ is a weak equivalence of $\Top_P$. In particular, it is enough to show that $\epsilon_X$ is a weak-equivalence of $\Top_{N(P)}$.
Consider the following commutative diagram
\begin{equation*}
\begin{tikzcd}[column sep = huge]
\fil{X}
\arrow[bend left= 18]{drr}{\Id_X}
\arrow{dr}{\epsilon_X}
\arrow[swap, bend right = 18]{dddr}{\varphi_X}
\\
&\fil{X}\times_P\Real{N(P)}
\arrow{r}{\pr_X}
\arrow{d}{\varphi_X\times_P\Real{N(P)}}
&\fil{X}
\arrow{d}{\varphi_X}
\\
&\Real{N(P)}\times_P\Real{N(P)}
\arrow[swap]{r}{\varphi_P\times_P\Real{N(P)}}
\arrow{d}
&\Real{N(P)}
\arrow{d}{\varphi_P}
\\
&\Real{N(P)}
\arrow[swap]{r}{\varphi_P}
&P
\end{tikzcd}
\end{equation*}
All three squares are pullback square. By Lemma \ref{LemmeVarphiPWeakEquivalence} $\varphi_P$ is a trivial fibration in $\Top_P$. Since $-\times_P\Real{N(P)}$ preserves trivial fibrations, $\varphi_P\times_P\Real{N(P)}$ is a trivial fibration in $\Top_{N(P)}$, and since the top square is a pullback square, so is $\pr_X$. But then, by two out of three, $\epsilon_X$ is a weak-equivalence in $\Top_{N(P)}$.
\end{proof}

\begin{lemme}\label{LemmeVarphiPWeakEquivalence}
The filtered map 
\begin{equation*}
\varphi_P\colon (\Real{N(P)},\varphi_P)\to (P,\Id_P)
\end{equation*}
is a trivial fibration in $\Top_P$.
\end{lemme}

\begin{proof}
We need to prove that for all $\Delta^{\varphi}\in R(P)$, the map
\begin{equation*}
\Sing(\C^0_P(\RealP{\Delta^{\varphi}},(\Real{N(P)},\varphi_P)))\to \Sing(\C^0_P(\RealP{\Delta^{\varphi}},(P,\Id_P)))
\end{equation*}
Is a trivial fibration in $\sS$. Using the Quillen-equivalence given by $(\Real{-},\Sing)$, this amounts to showing that the map
\begin{equation}\label{EquationVarphiPWeakEquivalence}
\C^0_P(\RealP{\Delta^{\varphi}},(\Real{N(P)},\varphi_P))\to \C^0_P(\RealP{\Delta^{\varphi}},(P,\Id_P))
\end{equation}
is a trivial fibration in $\Top$. First notice that $\C^0_P(\RealP{\Delta^{\varphi}},(P,\Id_P))\simeq \{*\}$, that is, there is only a single filtered map $\RealP{\Delta^{\varphi}}\to (P,\Id_P)$, given by the composition 
\begin{equation*}
\begin{tikzcd}
\RealP{\Delta^{\varphi}}
\arrow[hookrightarrow]{r}
&(\Real{N(P)},\varphi_P)
\arrow{r}{\varphi_P}
&(P,\Id_P)
\end{tikzcd}
\end{equation*}
In particular, since every topological space is fibrant, the map (\ref{EquationVarphiPWeakEquivalence}) is a fibration. To show that it is a weak-equivalence, it remains to show that $\C^0_P(\RealP{\Delta^{\varphi}},(\Real{N(P)},\varphi_P))$ is contractible. Write $\Delta^{\varphi}=\{p_0<\dots <p_n\}\subset N(P)$. First notice that any filtered map $f\colon \RealP{\Delta^{\varphi}}\to(\Real{N(P)},\varphi_P)$ must factor through $\varphi_P^{-1}(\{p_0,\dots ,p_n\})$. Indeed, since $f$ is filtered, one must have for any $x\in \RealP{\Delta^{\varphi}}$, $\varphi_P(f(x))=\varphi_P(x)\in \{p_0,\dots,p_n\}\subset P$. And so, by Lemma \ref{LemmeDeltaVarphiPartNP}, we get a homotopy equivalence 
\begin{equation*}
\C^0_P(\RealP{\Delta^{\varphi}},(\Real{N(P)},\varphi_P))\sim  \C^0_P(\RealP{\Delta^{\varphi}},\RealP{\Delta^{\varphi}}).
\end{equation*}
But the latter is contractible by lemma \ref{LemmeFilteredMapsDeltaVarphiContractible}, which concludes the proof. 
\end{proof}

\begin{lemme}\label{LemmeDeltaVarphiPartNP}
Let $\Delta^{\varphi}=\{p_0<\dots<p_n\}$ be a (non-degenerate) simplex of $N(P)$. The inclusion
\begin{equation*}
\RealP{\Delta^{\varphi}}\hookrightarrow \varphi^{-1}_P(\{p_0,\dots ,p_n\})
\end{equation*}
is a filtered homotopy equivalence.
\end{lemme}

\begin{proof}
Let  $x=(\Delta^{\psi},(t_0,\dots,t_m))$ be a point in $\varphi_P^{-1}(\{p_0,\dots,p_n\})\subset\Real{N(P)}$, with $\Delta^{\psi}=\{q_0<\dots <q_m\}$. Define $j=\max\{i\ |\ t_i\not=0\}$, then $\varphi_P(x)=q_j$, and so $q_j\in \{p_0,\dots,p_n\}$. For any non-degenerate simplex $\Delta^{\psi}=\{q_0<\dots <q_m\}$, write $I^{\psi}\subset \{0,\dots,m\}$ for the set of $i$ satisfying $q_i\in \{p_0,\dots,p_n\}$. Define the following homotopy
\begin{align*}
H^{\psi}\colon \RealP{\Delta^{\psi}}\cap\varphi_P^{-1}(\{p_0<\dots <p_n\})\times [0,1]&\to \RealP{\Delta^{\psi}}\cap\varphi_P^{-1}(\{p_0<\dots <p_n\})\\
((t_0,\dots,t_m),s)&\mapsto \left(H^{\psi}_1((t_0,\dots,t_m),s),\dots,H^{\psi}_m((t_0,\dots,t_m),s)\right)
\end{align*}
Where, for $0\leq i\leq m$, $H^{\psi}_i$ is defined as follows
\begin{align*}
H^{\psi}_i\colon\RealP{\Delta^{\psi}}\cap\varphi_P^{-1}(\{p_0,\dots ,p_n\})\times [0,1]&\to [0,1]\\
((t_0,\dots,t_m),s)&\mapsto \left\{\begin{array}{cl}
t_i(1-s) &\text{ if $i\not\in I^{\psi}$}\\
t_i\left(1+s\left(\frac{\sum_{j\not\in I^{\psi}}t_j}{\sum_{j\in I^{\psi}}t_j}\right)\right) &\text{ if $i\in I^{\psi}$}
\end{array}\right.
\end{align*}
One checks that $\sum_iH^{\psi}_i=1$, and that the $H^{\psi}_i$ are all non-negative, so that $H^{\psi}$ is well-defined. Furthermore, by construction, $H^{\psi}$ is a filtered map. In particular, its image lies in $\varphi_P^{-1}(\{p_0,\dots ,p_n\})$. Finally, $H^{\psi}$ is a homotopy between the identity, and a composition
\begin{equation*}
\RealP{\Delta^{\psi}}\cap\varphi_P^{-1}(\{p_0,\dots ,p_n\})\to \RealP{\Delta^{\psi}}\cap\RealP{\Delta^{\varphi}}\hookrightarrow \RealP{\Delta^{\psi}}\cap\varphi_P^{-1}(\{p_0,\dots ,p_n\})
\end{equation*}
Let us show that we can glue the $H^{\psi}$ together. Let $\Delta^{\mu}\subset\Delta^{\psi}$ be two non-degenerate simplices  of $N(P)$,
 and let $x\in \RealP{\Delta^{\mu}}\cap\varphi_P^{-1}(\{p_0,\dots,p_n\})$. Then, $x=(\Delta^{\psi},(t_0,\dots,t_m))=(\Delta^{\mu},(u_0,\dots,u_l))$, with $t_{\alpha(i)}=u_i$ for all $0\leq i\leq l$, and $t_i=0$ if $i\not\in \alpha \{0,\dots, l\}$, where $\alpha\colon \Delta^l\to \Delta^m$ is the face map corresponding to the inclusion $\Delta^{\mu}\subset\Delta^{\psi}$. Furthermore, $I^{\psi}\cap\alpha(\{0,\dots,l\})=\alpha(I^{\mu})$. Let $s\in[0,1]$. For $i\in \{0,\dots,l\}$, one computes :
\begin{itemize}
\item if $i\in I^{\mu}$, $H^{\mu}_i((u_0,\dots,u_l),s)=u_i\left(1+s\left(\frac{\sum_{j\not\in I^{\mu}}u_j}{\sum_{j\in I^{\mu}}u_j}\right)\right)$. On the other hand, $\alpha(i)\in I^{\psi}$, and one has $H^{\psi}_{\alpha(i)}((t_0,\dots,t_m),s)=t_{\alpha(i)}\left(1+s\left(\frac{\sum_{j\not\in I^{\psi}}t_j}{\sum_{j\in I^{\psi}}t_j}\right)\right)$. But since $t_{\alpha(j)}=u_j$ for all $j\in \{0,\dots,l\}$, and $k\not\in\alpha(\{0,\dots,l\})\Rightarrow t_k=0$, one can rewrite the latter equality as 
\begin{align*}
H^{\psi}_{\alpha(i)}((t_0,\dots,t_m),s)&=t_{\alpha(i)}\left(1+s\left(\frac{\sum_{j\not\in I^{\mu}}t_{\alpha(j)}}{\sum_{j\in I^{\mu}}t_{\alpha(j)}}\right)\right)\\
&=u_i\left(1+s\left(\frac{\sum_{j\not\in I^{\mu}}u_j}{\sum_{j\in I^{\mu}}u_j}\right)\right)\\
&=H^{\mu}_i((u_0,\dots,u_l),s).
\end{align*}
\item if $i\not\in I^{\mu}$, then $H^{\mu}_i((u_0,\dots,u_l),s)=u_i(1-s)$. On the other hand, $\alpha(i)\not\in I^{\psi}$, and $H^{\psi}_{\alpha(i)}((t_0,\dots,t_m),s)=t_{\alpha(i)}(1-s)=u_i(1-s)=H^{\mu}_i((u_0,\dots,u_l),s)$.
\end{itemize}
In all cases, one has $H^{\mu}_i(x,s)=H^{\psi}_{\alpha(i)}(x,s)$. Furthermore, if $j\not\in \alpha(\{0,\dots,l\})$, $t_j=0$, and $H^{\psi}_{j}(x,s)=0$. In conclusion, we have $H^{\mu}(x,s)=H^{\psi}(x,s)$, and the $H^{\psi}$ are compatible with faces. Since $\Real{N(P)}$ is a simplicial complex, this is enough to glue the $H^{\psi}$ together to obtain a map
\begin{equation*}
H\colon \varphi_P^{-1}(\{p_0,\dots,p_n\})\times [0,1]\to \varphi_P^{-1}(\{p_0,\dots,p_n\})
\end{equation*}
In particular, $H$ provides a filtered homotopy between the identity and the composition
\begin{equation*}
\varphi^{-1}_P(\{p_0,\dots,p_n\})\to \RealP{\Delta^{\varphi}}\hookrightarrow \varphi^{-1}_P(\{p_0,\dots,p_n\})
\end{equation*}
\end{proof}

\begin{lemme}\label{LemmeFilteredMapsDeltaVarphiContractible}
Let $\Delta^{\varphi}$ be a (non-degenerate) simplex of $N(P)$. The space of filtered maps
\begin{equation*}
\C^0_P(\RealP{\Delta^{\varphi}},\RealP{\Delta^{\varphi}})
\end{equation*}
is contractible.
\end{lemme}

\begin{proof}
Define the contracting homotopy as follows 
\begin{align*}
\C^0_P(\RealP{\Delta^{\varphi}},\RealP{\Delta^{\varphi}})\times [0,1]&\to \C^0_P(\RealP{\Delta^{\varphi}},\RealP{\Delta^{\varphi}})\\
\left(\left(f\colon \RealP{\Delta^{\varphi}}\to\RealP{\Delta^{\varphi}}\right), s\right)&\mapsto \left\{\begin{array}{rcl}
\RealP{\Delta^{\varphi}}&\to &\RealP{\Delta^{\varphi}}\\
(t_0,\dots,t_n)&\mapsto & s f(t_0,\dots,t_n)+ (1-s)(t_0,\dots,t_n)
\end{array}\right.
\end{align*}
\end{proof}
\section{A model category for stratified spaces}
\label{SectionCMFStrat}
In section \ref{SectionCMFTopP}, we constructed a model structure on the category of filtered spaces over any poset. In this section, using a theorem of P. Cagne and P.-A. Melliès \cite[Theorem 4.2]{Cagne}, we "glue" all those model structures together in order to obtain a model structure on $\Strat$ (Theorem \ref{TheoremCMFStrat}).

\subsection{Comparing model structures between posets}
\label{SubsectionQAAlpha}
For this section, fix some map of posets $\alpha\colon P\to Q$.

\begin{defin}
Define the functor $\alpha_*\colon \Top_P\to \Top_Q$, and $\alpha^*\colon \Top_Q\to\Top_P$ as follows
\begin{align*}
\alpha_*\colon \Top_P&\to\Top_Q\\
(X,P,\varphi_X)&\mapsto (X,Q,\alpha\circ\varphi_X)\\
\phantom{X}&\phantom{X}\\
\alpha^*\colon\Top_Q&\to\Top_P\\
(Y,Q,\varphi_Y)&\mapsto (Y\times_QP,P,\alpha^*\varphi_Y)
\end{align*}
where $\alpha^*\varphi_Y$ is defined by the following pullback square
\begin{equation*}
\begin{tikzcd}
Y\times_QP
\arrow{r}{\alpha^*\varphi_Y}
\arrow[swap]{d}{\pr_Y}
&P
\arrow{d}{\alpha}
\\
Y
\arrow[swap]{r}{\varphi_Y}
&Q
\end{tikzcd}
\end{equation*}
\end{defin}

\begin{prop}\label{PropositionQETopPTopQ}
The pair of functors 
\begin{equation*}
\alpha_*\colon \Top_P\leftrightarrow\Top_Q\colon \alpha^*
\end{equation*}
 is a Quillen-adjunction. Furthermore, it is a Quillen-equivalence if and only if $\alpha$ is an isomorphism of posets.
\end{prop}

\begin{proof}
It is an adjunction by construction. To show that it is a Quillen-adjunction, it is enough to show that the generating (trivial) cofibrations of $\Top_P$ are sent to (trivial) cofibrations. But 
\begin{equation}\label{EquationAlphaCofibration}
\alpha_*(\RealP{\partial(\Delta^n)\otimes\Delta^{\varphi}\to\Delta^n\otimes\Delta^{\varphi}})\simeq\Real{\partial(\Delta^n)\otimes\Delta^{\alpha\circ\varphi}\to\Delta^n\otimes\Delta^{\alpha\circ\varphi}}_Q.
\end{equation} If $\Delta^{\alpha\circ\varphi}$ is a non-degenerate simplex of $N(Q)$, (\ref{EquationAlphaCofibration}) is a generating cofibration of $\Top_Q$ , if not, it is the retract of a cofibration by Lemma \ref{LemmeRetractNonDegenerateSimplex}. The same argument shows that $\alpha_*$ preserves trivial cofibrations, and so $\alpha_*$ is a left Quillen functor. 

If $\alpha$ is an isomorphism, $\alpha_*$ and $\alpha^*$ are inverse equivalences of categories. Furthermore, $\alpha^*$ and $\alpha_*$ map the set of generating (trivial) cofibrations of $\Top_P$ and $\Top_Q$ to each other (up to isomorphism), which means that $(\alpha_*,\alpha^*)$ is a Quillen-equivalence. Conversely, suppose $p_0\not=p_1$ satisfy $\alpha(p_0)=\alpha(p_1)=q\in Q$. Consider the stratified space $X=(\{*\},P,*\mapsto p_0)$. It is cofibrant, and one has $\alpha^*\alpha_*(X)=(\alpha^{-1}\{q\},P,\alpha^{-1}(q)\subset P)$. In particular, $s\pi_0(\alpha^*\alpha_*(X))(\{p_1\})\not = \emptyset = s\pi_0(X)$, and so $X\to \alpha^*\alpha_*(X)$ is not a weak-equivalence. Next assume that there exist $q\in Q$ not in the image of $\alpha$, and let $Y=(\{*\},Q,*\mapsto q)$. Then, $\alpha_*\alpha^*(Y)$ is empty. In particular, $\alpha_*\alpha^*(Y)\to Y$ is not a weak-equivalence. Finally, assume that $\alpha$ is a bijection and that there exist $p_0,p_1\in P$ such that $p_0\not<p_1$ but $\alpha(p_0)=q_0<q_1=\alpha(p_1)$. Consider $X=(\{p_0,p_1\},P,\{p_0,p_1\}\hookrightarrow P)$ in $\Top_P$ and $Y=(\Real{\Delta^{\varphi}}_Q,Q,\varphi_Q)$ in $\Top_Q$, where $\Delta^{\varphi}=\{q_0<q_1\}$ is a non-degenerate simplex of $N(Q)$. Then, consider the map $i\colon\alpha_*(X)\to Y$ which corresponds to the inclusion $\Real{\partial(\Delta^{\varphi})}_Q\to \Real{\Delta^{\varphi}}_Q$. Its adjoint is a weak-equivalence, since $D_P(X)(p_0)\sim \{*\}\sim D_P(\alpha^*(Y))(p_0)$ and $D_P(X)(p_1)\sim \{*\}\sim D_P(\alpha^*(Y))$. But $i$ is not a weak-equivalence since $D_Q(\alpha_*(X))(\Delta^{\varphi})=\emptyset$ but $D_Q(Y)(\Delta^{\varphi})\not= \emptyset$.
\end{proof}

\begin{lemme}\label{LemmeRetractNonDegenerateSimplex}
Let $\varphi\colon \Delta^{n}\to N(P)$ be a not-necessarily injective simplicial map. Then, there exists an injective map $\bar{\varphi}\colon\Delta^k\to N(P)$ and a topological space $K$ such that $\RealP{\Delta^{\varphi}}$ is a retract of $\RealP{\Delta^{\bar{\varphi}}}\otimes K$
\end{lemme}

\begin{proof}
Let $\bar{\varphi}\colon \Delta^k\to N(P)$ be the inclusion $\Im(\varphi)\to N(P)$. Enumerate the vertices of $\Delta^k\in N(P)$ as $\{p_0<\dots<p_k\}$. Enumerate the vertices of $\Delta^n$ as $\{e_0^{p_0},e_1^{p_0},\dots,e_{m_0}^{p_0},e_0^{p_1},\dots,e_{m_k}^{p_k}\}$, where $\varphi(e_i^{p_j})=p_j$. Consider the map defined on vertices as
\begin{align*}
f\colon\Delta^k\times\Delta^{n}&\to\Delta^n\\
(p_l,e_i^{p_j})&\mapsto \left\{\begin{array}{cl}
e_i^{p_j}& \text{ if $p_j=p_l$}\\
e_0^{p_l}& \text{ if $p_j\not = p_l$}
\end{array}\right.
\end{align*}
It extends to a map of simplicial sets. We then have the following retract
\begin{equation*}
\begin{tikzcd}
\Delta^n
\arrow{r}{(\varphi,\Id)}
\arrow{d}{\varphi}
&\Delta^k\times\Delta^n
\arrow{d}{\bar{\varphi}\circ \pr_{\Delta^k}}
\arrow{r}{f}
&\Delta^n
\arrow{d}{\varphi}
\\
N(P)
\arrow{r}{=}
&N(P)
\arrow{r}{=}
&N(P)
\end{tikzcd}
\end{equation*}
Taking the realization gives the desired retract.
\end{proof}

\subsection{The model category $\Strat$}
\label{SubsectionCMFStrat}
\begin{defin}
Let $f\colon (X,P,\varphi_X)\to (Y,Q,\varphi_Y)$ be a stratified map, such that $\widehat{f}\colon P\to Q$ is an isomorphism of posets. We will also write $\widehat{f}$ for the corresponding isomorphism $\widehat{f}\colon N(P)\to N(Q)$. Define the functor $s\pi_0(Y,Q,\varphi_Y)\circ\widehat{f}$ as follows
\begin{align*}
s\pi_0(Y,Q,\varphi_Y)\circ\widehat{f}\colon R(P)^{\op}&\to \Set\\
\Delta^{\psi}&\mapsto s\pi_0(Y,Q,\varphi_Y)(\Delta^{\widehat{f}\circ\psi})
\end{align*}
Let $\phi\colon \RealP{\Delta^{\varphi}}\to (X,P,\varphi_X)$ be a pointing, and $n\geq 1$ an integer. Define the pointing of $(Y,Q,\varphi_Y)$, $f\circ \phi$ as
\begin{equation*}
f\circ\phi\colon \Real{\Delta^{\widehat{f}\circ\varphi}}_Q\to (Y,Q,\varphi_Y)
\end{equation*}
Then, define the functor $s\pi_n((Y,Q,\varphi_Y),f\circ \phi)\circ\widehat{f}$ as follows
\begin{align*}
s\pi_n((Y,Q,\varphi_Y),f\circ\phi)\circ \widehat{f}\colon R(\Im(\varphi))^{\op}&\to \Set\\
\Delta^{\psi}&\mapsto s\pi_n((Y,Q,\varphi_Y),f\circ\phi)(\Delta^{\widehat{f}\circ\psi})
\end{align*}
\end{defin}

\begin{prop}
Let $f\colon (X,P,\varphi_X)\to (Y,Q,\varphi_Y)$ be a stratified map, such that $\widehat{f}\colon P\to Q$ is an isomorphism of posets. Then, $f$ induces a morphism of functors
\begin{equation*}
s\pi_0(f)\colon s\pi_0(X,P,\varphi_X)\to s\pi_0(Y,Q,\varphi_Y)\circ\widehat{f}
\end{equation*}
Furthermore, for all pointing of $(X,P,\varphi_X)$, $\phi$, and all $n\geq 1$, $f$ induces morphisms of functors
\begin{equation}\label{EquationSPiNFChapeau}
s\pi_n(f)\colon s\pi_n((X,P,\varphi_X),\phi)\to s\pi_n((Y,Q,\varphi_Y),f\circ\phi)\circ\widehat{f}
\end{equation}
\end{prop}
\begin{proof}
Define $s\pi_n(f)$ as follows. An element in $s\pi_n((X,P,\varphi_X),\phi)(\Delta^{\varphi})$ is an element in 
\begin{equation*}
\pi_n\left(\Sing\left(\C^0_P\left(\RealP{\Delta^{\varphi}},(X,P,\varphi_X)\right)\right),\phi_{|\RealP{\Delta^{\varphi}}}\right)
\end{equation*}
Using the isomorphism $\alpha\colon P\to Q$ this is in bijection with
\begin{equation}\label{EquationPiNcircf}
\pi_n\left(\Sing\left(\C^0_Q\left(\Real{\Delta^{\alpha\circ\varphi}}_Q,(X,Q,\alpha\circ\varphi_X)\right)\right),\phi_{|\Real{\Delta^{\alpha\circ\varphi}}_Q}\right)
\end{equation}
the map $f$ induces a map $f_{\triangleright}\colon (X,P,\alpha\circ\varphi_X)\to(Y,\varphi_Y,Q)$ in $\Top_Q$ (see Proposition \ref{PropositionTriangleLeftRight}) which in turns sends an element of (\ref{EquationPiNcircf}) to an element in
\begin{equation*}
\pi_n\left(\Sing\left(\C^0_Q\left(\Real{\Delta^{\alpha\circ\varphi}}_Q,(Y,Q,\varphi_Y)\right)\right),f\circ\phi_{|\Real{\Delta^{\alpha\circ\varphi}}_Q}\right)
\end{equation*}
The composition gives us a map
\begin{equation*}
\begin{tikzcd}[column sep = large]
s\pi_n((X,P,\varphi_X),\phi)(\Delta^{\varphi})
\arrow[swap]{d}{\simeq}
\arrow{dr}{s\pi_n(f)}
\\
s\pi_n((X,Q,\alpha\circ\varphi_X),\phi)(\Delta^{\alpha\circ\varphi})
\arrow[swap]{r}{s\pi_n(f_{\triangleright})}
&s\pi_n((Y,Q,\varphi_Y),f\circ\phi)(\Delta^{\alpha\circ\varphi})
\end{tikzcd}
\end{equation*}
Which is exactly the map (\ref{EquationSPiNFChapeau}) evaluated at $\Delta^{\varphi}$.
\end{proof}

The goal of this section is to prove the following theorem

\begin{theo}\label{TheoremCMFStrat}
There exists a cofibrantly generated model structure on $\Strat$ such that a stratified map $f\colon (X,P,\varphi_X)\to (Y,Q,\varphi_Y)$ is a weak-equivalence if and only if $\widehat{f}\colon P\to Q$ is an isomorphism of poset, and if the map
\begin{equation*}
s\pi_0(f)\colon s\pi_0(X,P,\varphi_X)\to s\pi_0(Y,Q,\varphi_Y)\circ\widehat{f}
\end{equation*}
is an isomorphism, and if
for all pointings of $X$, $\phi$, and all $n\geq 1$, the map
\begin{equation*}
s\pi_n(f)\colon s\pi_n((X,P,\varphi_X),\phi)\to s\pi_n((Y,Q,\varphi_Y),f\circ \phi)\circ \widehat{f}
\end{equation*}
is an isomorphism. 
The generating set of cofibration and trivial cofibrations are the following
\begin{itemize}
\item $I=\{ S^{n-1}\times\Real{\Delta^{\varphi}}_{\mathbb{N}}\to D^{n}\times \Real{\Delta^{\varphi}}_{\mathbb{N}}\ | n\geq 0, \varphi\colon \Delta^k\to N(\mathbb{N})\}$
\item $J=\{D^n\times \Real{\Delta^{\varphi}}_{\mathbb{N}}\to D^n\times [0,1]\times \Real{\Delta^{\varphi}}_{\mathbb{N}}\ |\ n\geq 0, \varphi\colon \Delta^k\to N(\mathbb{N})\}$
\end{itemize}
\end{theo}

In order to prove Theorem \ref{TheoremCMFStrat}, we will need the following :
\begin{prop}\label{PropositionTriangleLeftRight}
Let $f\colon (X,P,\varphi_X)\to(Y,Q,\varphi_Y)$ be a stratified map, and write $\alpha=\widehat{f}\colon P\to Q$. The stratified map $f$ admits the following two factorisations :
\begin{equation*}
\begin{tikzcd}
(X,P,\varphi_X)
\arrow{r}
\arrow[swap]{d}{f^{\triangleleft}}
\arrow{dr}{f}
&(X,Q,\alpha\circ\varphi_X)
\arrow{d}{f_{\triangleright}}
\\
(Y\times_Q P,P,\alpha^*\varphi_Y)
\arrow{r}
&(Y,Q,\varphi_Y)
\end{tikzcd}
\end{equation*}
where $f^{\triangleleft}$ is a map in $\Top_P$ and $f_{\triangleright}$ is a map in $\Top_Q$.
\end{prop}

\begin{proof}
The filtered map $f_{\triangleright}$ is simply $f\colon X\to Y$ together with $\Id_Q$, and we clearly have $(f,\alpha)=(\Id_X,\alpha)\circ(f,\Id_Q)$. We obtain the other half of the square by adjunction.
\end{proof}

We can now define the following classes of maps.
\begin{defin}\label{DefinitionClassesOfStrat}
A stratified map $f\colon (X,P,\varphi_X)\to (Y,Q,\varphi_Y)$ is :
\begin{itemize}
\item a stratified cofibration if $f_{\triangleright}$ is a cofibration in $\Top_Q$.
\item a stratified fibration if $f^{\triangleleft}$ is a fibration in $\Top_P$.
\item a stratified trivial cofibration if $f_{\triangleright}$ is a trivial cofibration in $\Top_Q$.
\item a stratified trivial fibration if $f^{\triangleleft}$ is a trivial fibration in $\Top_P$.
\end{itemize}
\end{defin}

\begin{proof}[Proof of Theorem \ref{TheoremCMFStrat}]
Consider the functor
\begin{align*}
F\colon \Strat&\to \text{Poset}\\
(X,P,\varphi_X)&\mapsto P
\end{align*}
where $\text{Poset}$ is the category of posets and order preserving maps.
It is a Grothendieck bifibration \cite[Definition 2.1]{Cagne}. Consider the trivial model structure on $\text{Poset}$, where weak-equivalences are the isomorphisms, and all maps are cofibrations and fibrations. By proposition \ref{PropositionQETopPTopQ} hypothesis $(Q)$ of \cite[Theorem 4.2]{Cagne} is verified, and the other hypothesis follow easily from the fact that $\text{Poset}$ is taken with the trivial model structure. Applying \cite[Theorem 4.2]{Cagne} to $F$ gives a model structure on $\Strat$ with the classes of (trivial) cofibrations and fibrations of definition \ref{DefinitionClassesOfStrat}. Let us show that it coincides with the model structure described in Theorem \ref{TheoremCMFStrat}. Let $f\colon (X,P,\varphi_X)\to (Y,Q,\varphi_Y)$ be a stratified map, and factor it as a trivial cofibration followed by a fibration. 
\begin{equation*}
\begin{tikzcd}
(X,P,\varphi_X)
\arrow{r}{j}
&(Z,R,\varphi_Z)
\arrow{r}{p}
&(Y,Q,\varphi_Y)
\end{tikzcd}
\end{equation*}
Write $\widehat{j}=\beta$ and $\widehat{p}=\gamma$, and consider the following commutative diagram :
\begin{equation*}
\begin{tikzcd}
(X,P,\varphi_X)
\arrow[bend right = 80, swap]{dd}{f^{\triangleleft}}
\arrow{r}
\arrow{d}{j^{\triangleleft}}
\arrow{dr}{j}
&(X,R,\beta\circ\varphi_X)
\arrow{r}
\arrow{dr}
\arrow{d}{j_{\triangleright}}
&(X,Q,\alpha\circ\varphi_X)
\arrow{d}
\arrow[bend left = 80]{dd}{f_{\triangleright}}
\\
(Z\times_RP,P,\beta^*\varphi_Z)
\arrow{r}
\arrow{dr}
\arrow{d}
&(Z,R,\varphi_Z)
\arrow{r}
\arrow{dr}{p}
\arrow{d}{p^{\triangleleft}}
&(Z,Q,\gamma\circ\varphi_Z)
\arrow{d}{p_{\triangleright}}
\\
(Y\times_QP,P,\alpha^*\varphi_Y)
\arrow{r}
&(Y\times_QR,R,\gamma^*\varphi_Y)
\arrow{r}
&(Y,Q,\varphi_Y)
\end{tikzcd}
\end{equation*}

$f$ is a weak-equivalence if and only if $p$ is a trivial fibration. The latter is true if and only if $\gamma$ is an isomorphism and if $p^{\triangleleft}$ is a trivial fibration. If $\gamma$ is an isomorphism, all rows are weak-equivalences. By repeated applications of the two out of three property, we get that $f$ is a weak-equivalence if and only if all rows and columns are weak-equivalences. In particular, $f$ is a weak-equivalenc if and only if $f^{\triangleleft}$ is a weak-equivalence in $\Top_P$. Lemma \ref{LemmeFTriangleLeftWeakEquivalence} then gives the desired characterization of weak-equivalences.

It remains to show that $I$ and $J$ are sets of generating (trivial) cofibrations.
Let $f\colon (X,P,\varphi_X)\to(Y,Q,\varphi_Y)$ be a stratified map, and consider the following lifting problem
\begin{equation}\label{DiagramGeneratingCofibrations}
\begin{tikzcd}
S^{n-1}\times\Real{\Delta^{\varphi}}_{\mathbb{N}}
\arrow{d}
\arrow{r}{b}
&(X,P,\varphi_X)
\arrow{d}{f}
\\
D^{n}\times \Real{\Delta^{\varphi}}_{\mathbb{N}}
\arrow{r}{c}
\arrow[dashrightarrow]{ur}
&(Y,Q,\varphi_Y)
\end{tikzcd}
\end{equation}
Write $\widehat{f}=\alpha$, $\widehat{b}=\beta\colon \N\to P$ and $\widehat{c}=\gamma\colon \N\to Q$. The commutativity of the diagram (\ref{DiagramGeneratingCofibrations}) guarantees that $\alpha\circ\beta= \gamma$. Factoring $c$ as follows
\begin{equation*}
\begin{tikzcd}
D^{n}\times \Real{\Delta^{\varphi}}_{\mathbb{N}}
\arrow{r}{d}
&D^{n}\times \RealP{\Delta^{\alpha\circ\varphi}}
\arrow{r}{e}
&(Y,Q,\varphi_Y)
\end{tikzcd}
\end{equation*}
where $\widehat{d}=\beta$ and $\widehat{e}=\alpha$.
We can now use adjunction $(\alpha_*,\alpha^*)$, to get a lifting problem in $\Top_P$, equivalent to (\ref{DiagramGeneratingCofibrations}). 
\begin{equation*}
\begin{tikzcd}
S^{n-1}\times\RealP{\Delta^{\alpha\circ\varphi}}
\arrow{d}
\arrow{r}{b_{\triangleright}}
&(X,P,\varphi_X)
\arrow{d}{f^{\triangleleft}}
\\
D^{n}\times \RealP{\Delta^{\alpha\circ\varphi}}
\arrow{r}{e^{\triangleleft}}
\arrow[dashrightarrow]{ur}
&(Y\times_Q P,P,\alpha^*\varphi_Y)
\end{tikzcd}
\end{equation*}
The maps $S^{n-1}\times\RealP{\Delta^{\alpha\circ\varphi}}\to D^{n}\times \RealP{\Delta^{\alpha\circ\varphi}}$ are not necessarily generating cofibration since $\alpha\circ\varphi\colon \Delta^k\to N(P)$ might not be injective, but by lemma \ref{LemmeRetractNonDegenerateSimplex} $f^{\triangleleft}$ is a trivial fibration of $\Top_P$ if and only if it has the right lifting property against all maps of this form. In particular, $f$ is a trivial fibration if and only if it has the right lifting property against all maps of $I$. The proofs works similarly for the set of generating trivial cofibrations.
\end{proof}

\begin{lemme}\label{LemmeFTriangleLeftWeakEquivalence} Let $f\colon(X,P,\varphi_X)\to (Y,Q,\varphi_Y)$ be a stratified map such that $\widehat{f}=\alpha\colon P\to Q$ is an isomorphism of posets. Then $f^{\triangleleft}\colon (X,P,\varphi_X)\to (Y\times_QP,P,\alpha^*\varphi_Y)$ is a weak-equivalence in $\Top_P$ if and only if $f$ induces isomorphisms
\begin{equation*}
s\pi_0(f)\colon s\pi_0(X,P,\varphi_X)\to s\pi_0(Y,Q,\varphi_Y)\circ\widehat{f}
\end{equation*}
and, 
\begin{equation*}
s\pi_n(f)\colon s\pi_n((X,P,\varphi_X),\phi)\to s\pi_n((Y,Q,\varphi_Y),f\circ\phi)\circ\widehat{f}
\end{equation*}
for all pointing of $(X,P,\varphi_X)$, $\phi$, and all $n\geq 1$.
\end{lemme}

\begin{proof}
The map induced by $f$ on stratified homotopy groups factors in two ways as follows
\begin{equation*}
\begin{tikzcd}[column sep = large]
s\pi_n((X,P,\varphi_X),\phi)
\arrow[swap]{d}{\simeq}
\arrow{r}{s\pi_n(f^{\triangleleft})}
\arrow[swap]{dr}{s\pi_n(f)}
&s\pi_n((Y\times_QP,P,\alpha^*\varphi_Y),f^{\triangleleft}\circ\phi)
\arrow{d}{s\pi_n(\pr_Y)}
\\
s\pi_n((X,P,\alpha\circ\varphi_X),\phi)\circ\alpha
\arrow[swap]{r}{s\pi_n(f_{\triangleright})}
&s\pi_n((Y,Q,\varphi_Y),f\circ\phi)\circ\alpha
\end{tikzcd}
\end{equation*}
Now, $f^{\triangleleft}$ is a weak-equivalence if and only if $s\pi_n(f^{\triangleleft})$ is an isomorphism for all $n$ and $\phi$. But the latter is true if and only if $s\pi_n(f)$ is an isomorphism for all $n$ and $\phi$, since $s\pi_n(\pr_Y)$ is an isomorphism.
\end{proof}

\begin{remarque}
The proof of Theorem \ref{TheoremCMFStrat} - and of Proposition \ref{PropositionQETopPTopQ} on which it depends - generalizes without difficulty to the category $\StrStrat$ of strongly stratified spaces or even to a category of strongly stratified simplicial sets. In particular, there exists a cofibrantly generated model category of strongly stratified spaces.
\end{remarque}

\appendix

\section{A property of some colimits over particular posets}
\label{AppendixColimits}

\begin{defin}\label{DefinitionWellBehaved}
Let $\D$ be a poset and $G\colon \D\to \Set$ be a functor. We say that $\D$ is almost filtered with respect to $G$ if, 
\begin{enumerate}
\item 
 for all $d\in \D$ and all $d_1,d_2,d_3\in \D$ such that $d_1\leq d,d_2$, and $d_3\leq d_2,d$, 
if there exist $x,y\in G(d)$ and $x_i\in G(d_i)$, $i\in \{1,2,3\}$, such that $G(d_1\to d)(x_1)=x$, $G(d_1\to d_2)(x_1)=x_2$, $G(d_3\to d_2)(x_3)=x_2$ and $G(d_3\to d)(x_3)=y$, then
there exists $e\in \D$ such that the following relations are satisfied,
\begin{equation}\label{EquationZigzagShapeN2}
\begin{tikzcd}
d
&d_1
\arrow{dr}
\arrow{l}
\arrow{r}
&d_2
&d_3
\arrow{dl}
\arrow{l}
\arrow{r}
&d
\\
&&e
\arrow{ull}
\arrow{u}
\arrow{urr}
\end{tikzcd}
\end{equation}
\item
\begin{itemize}
\item for all $n\geq 2$, and all $(2n+3)$-tuple in $\D$, $(d_0,\dots,d_{2n+2})$, with $d_0=d_{2n+2}=d$, and $d_{2k+1}\leq d_{2k},d_{ 2k+2}$, for all $0\leq k\leq n$, 
\item for all $(2n+3)$-tuple $(x_0,\dots,x_{2n+2})$, such that $x_i\in G(d_i)$, for $0\leq i\leq 2n+2$, and $G(d_{2k+1}\to d_{2k})(x_{2k+1})=x_{2k}$ and $G(d_{2k+1}\to d_{2k+2})(x_{2k+1})=x_{2k+2}$ for all $0\leq k\leq n$.
\end{itemize}
there exists :
\begin{itemize}
\item a $(2n+3)$-tuple of elements in $\D$, $(d'_0,\dots,d'_{2n+2})$ and $e$ an element in $\D$, such that $d'_0=d'_{2n+1}=d$, and such that the relation depicted in the following diagram are satisfied
\begin{equation}\label{EquationAlmostFilteredPoset}
\begin{tikzcd}
&&&\phantom{X}e\phantom{X}
\\
&d'_1
\arrow{dl}
\arrow{r}
\arrow{urr}
\arrow{dd}
&d'_2
\arrow{ur}
\arrow{dd}
&d'_3
\arrow{u}
\arrow{l}
\arrow{r}
\arrow{dd}
&\dots
\arrow[phantom]{ul}{\dots}
&d'_{2n+1}
\arrow{ull}
\arrow{l}
\arrow{dd}
\arrow{dr}
\\
d
&&&&&&d
\\
&d_1
\arrow{ul}
\arrow{r}
&d_2
&d_3
\arrow{l}
\arrow{r}
&\dots
&\arrow{l}
d_{2n+1}
\arrow{ur}
\end{tikzcd}
\end{equation}
\item a $(2n+3)$-tuple $(x'_0,\dots,x'_{2n+2})$ such that $x'_i\in G(d'_i)$ for all $0\leq i\leq 2n+2$ and $G(d'_{2k+1}\to d'_{2k})(x'_{2k+1})=x'_{2k}$ and $G(d'_{2k+1}\to d'_{2k+2})(x'_{2k+1})=x'_{2k+2}$ for all $0\leq k\leq n$.
\end{itemize}
such that $G(d'_i\to d_i)(x'_i)=x_i$ for all $0\leq i\leq 2n+2$.
\end{enumerate}
\end{defin}

\begin{prop}\label{PropositionColimiteMono} Let $\D$ be a poset, and $G\colon \D\to \Set$ a functor sending morphisms of $\D$ to monomorphisms in $\Set$. If $\D$ is almost filtered with respect to $G$, then for any $d\in \D$, the map 
\begin{equation*}
G(d)\to\colim\limits_{d\in \D}(G)
\end{equation*}
is a monomorphism.
\end{prop}

\begin{proof}
Recall the following
\begin{equation*}
\colim_{\D}G=\left(\coprod_{d\in \D}G(d)\right)/{\sim}
\end{equation*}
Where ${\sim}$ is the equivalence relation generated by $x\sim y$ if $x\in G(c)$, $y\in G(d)$ and $G(c\to d)(x)=y$.
It suffices to prove that if $d\in \D$ and $x,y$ are two elements of $G(d)$, then 
\begin{equation}
x\sim y \Leftrightarrow x=y.
\end{equation} Assume that $x\sim y$. Then, there exists a zigzag of maps in $\D$, realizing this equivalence, of one of the following three forms
\begin{equation}\label{Zigzag1}
\begin{tikzcd}
d
&d_1
\arrow{l}
\arrow{r}
&\dots
&d_{2n+1}
\arrow{l}
\arrow{r}
&d
\end{tikzcd}
\end{equation}
\begin{equation}\label{Zigzag2}
\begin{tikzcd}
d
&d_1
\arrow{l}
\arrow{r}
&d_2
&\dots
\arrow{l}
\arrow{r}
&d_{2n}
&d
\arrow{l}
\end{tikzcd}
\end{equation}
\begin{equation}\label{Zigzag3}
\begin{tikzcd}
d
\arrow{r}
&d_1
&\dots
\arrow{l}
\arrow{r}
&d_{2n+1}
&d
\arrow{l}
\end{tikzcd}
\end{equation}
Composing the first map of (\ref{Zigzag2}) with $d\to d_{2n}$, and omitting the last gives a zigzag of shape (\ref{Zigzag1}) between $G(d\to d_{2n})(x)$ and $G(d\to d_{2n})(y)$. Since $G(d\to d_{2n})$ is a monomorphism, showing that $x=y$ is equivalent to showing that the two former elements of $G(d_{2n})$ are equal. Similarly, omitting the first map in (\ref{Zigzag3}) and adding the map $d\to d_1$ at the end gives a zigzag of shape (\ref{Zigzag1}) between $G(d\to d_1)(x)$ and $G(d\to d_1)(y)$. Since $G(d\to d_1)$ is a monomorphism by assumption, those elements of $G(d_1)$ are equal if and only if $x=y$. In particular, it is enough to show that the existence of a zigzag of shape (\ref{Zigzag1}) realizing an equivalence $x\sim y$, implies that $x=y$. Fix such a zigzag, and elements $x_i\in G(d_i)$, for all $0\leq i\leq 2n+2$ (take $x_0=x$, and $x_{2n+2}=y$), such that $G(d_{2k+1}\to d_{2k})(x_{2k+1})=x_{2k}$ and $G(d_{2k+1}\to d_{2k+2})(x_{2k+1})=x_{2k+2}$.

Assume that $n\geq 2$. Since $\D$ is almost filtered with respect to $G$, there exist  $e\in \D$ and a tuple, $(d'_1,\dots,d'_{2n+1})$ such that the relations of (\ref{EquationAlmostFilteredPoset}) hold. Furthermore, there exists $(x'_1,\dots,x'_{2n+1})$ such that the relation of Definition \ref{DefinitionWellBehaved} are satisfied. The commutativity of (\ref{EquationAlmostFilteredPoset}) guarantees that $x_e=G(d'_i\to e)(x'_i)$ is well defined and independant of the choice of $i$. In particular, we have a zigzag of the form
\begin{equation*}
\begin{tikzcd}
d
&d'_1
\arrow{l}
\arrow{r}
&e
&d'_{2n+1}
\arrow{l}
\arrow{r}
&d
\end{tikzcd}
\end{equation*}
realizing the equivalence $x\sim y$. This is exactly a zigzag of shape (\ref{Zigzag1}), with $n=1$.

Assume that $n=1$. Since $\D$ is almost filtered with respect to $G$, there exists $e\in \D$ such that there is a diagram of shape (\ref{EquationZigzagShapeN2}). Taking $x_e=G(d_1\to e)(x_1)=G(d_3\to e)\in G(e)$, we get a zigzag of shape (\ref{Zigzag1}), with $n=0$.

Assume that $n=0$. This means that there exists $d_1\in \D$, and $x_1$ such that $G(d_1\to d)(x_1)=x$ and $G(d_1\to d)(x_1)=y$. Clearly, this implies that $x=y$.
\end{proof}

\begin{remarque}
At fist glance, one might think that the hypothesis in Proposition \ref{PropositionColimiteMono} that $\D$ is almost filtered with respect to $G$ is unnecessary , and that the conclusion should hold for any functor $G\colon \D\to \Set$ sending morphisms of $\D$ to monomorphisms, as long as $\D$ is a poset. But this is not true. Consider the poset $\D$ with four objects $a,b,c$ and $d$ and relations $a\to c$, $a\to d$, $b\to c$, $b\to d$, and the functor $G\colon \D\to \Set$ defined by $G(a)=G(b)=G(c)=G(d)=\{0,1\}$, and $G(a\to c)=G(a\to d)=G(b\to c)=\Id$, and $G(b\to d)= 1-\Id$. Then, $\colim_{\D}G=\{*\}$.
\end{remarque}

\begin{remarque}\label{RemarqueSeenAsSet}
One can see any (filtered) simplicial set as a set by considering its set of simplices. Indeed, if $(X,\varphi_X)$ is a filtered simplicial set, and $\Delta^{\varphi}\in N(P)$ is a simplex of $N(P)$, write $X_{\Delta^{\varphi}}$ for the set of $\Delta^{\varphi}$-simplices of $X$ (equivalently, for the set of filtered maps $\Hom_{\sS_P}(\Delta^{\varphi},(X,\varphi_X))$. The set of simplices of $X$ is then the disjoint union $U(X)=\coprod\limits_{\Delta^{\varphi}\in\Delta(P)} X_{\Delta^{\varphi}}$. Notice that a map of (filtered) simplicial sets $(X,\varphi_X)\to(Y,\varphi_Y)$ is a monomorphism if and only if the corresponding map of sets $U(X)\to U(Y)$ is a monomorphism.
In the same way, one can see any functor with values in the category of (filtered) simplicial sets, as a functor with values in $\Set$ by composing it with the forgetful functor $U$. Further, let $\C$ be a small category and $F\colon \C\to \sS_P$ be a functor, for all $\Delta^{\varphi}\in\Delta(P)$, one defines the functor $F_{\Delta^{\varphi}}\colon \C\to\sS_P$ by taking $F_{\Delta^{\varphi}}(c)=F(c)_{\Delta^{\varphi}}$ for all $c\in\C$. One then has:
\begin{equation*}
\left(\colim\limits_{c\in\C}F(c)\right)_{\Delta^{\varphi}}=\colim\limits_{c\in\C}\left(F_{\Delta^{\varphi}}(c)\right)
\end{equation*}
Equivalently, one says that the limits and colimits in $\sS_P$ are computed degree-wise. Now since $U\circ F$ is a colimit of the $F_{\Delta^{\varphi}}$, and since colimits commute with each-other, one has :
\begin{equation*}
\colim\limits_{\C} (U\circ F) = U\left(\colim_{\C} F\right).
\end{equation*}
In particular, for the matter of studying colimits of a functor with values in $\sS_P$, one can see it as a functor with values in $\Set$ by composing with the forgetful functor $U$. Furthermore, since $U$ reflects monomorphisms, if Proposition \ref{PropositionColimiteMono} applies to the functor with values in $\Set$, it will also apply to the original functor.
\end{remarque}

\begin{lemme}\label{LemmeCAlmostFiltered}
Let $F$ be an object in $\Diag_P$ for some poset $P$. Then $\mathcal{C}$ is almost filtered with respect to $F\otimes R(P)$, seen as a functor to $\Set$. (See Definition \ref{DefinitionColim}.)

\end{lemme}

\begin{proof}
Consider a diagram in $\C$ of shape
\begin{equation*}
\begin{tikzcd}[column sep = small]
(\Delta^{\varphi},\Delta^{\psi})
&(\Delta^{\varphi_1},\Delta^{\psi_1})
\arrow{l}
\arrow{r}
&(\Delta^{\varphi_2},\Delta^{\psi_2})
&(\Delta^{\varphi_3},\Delta^{\psi_3})
\arrow{l}
\arrow{r}
&(\Delta^{\varphi},\Delta^{\psi})
\end{tikzcd}
\end{equation*}
And simplices $\sigma,\tau\in F(\Delta^{\varphi})\otimes\Delta^{\psi}$, and $\sigma_i\in F(\Delta^{\varphi_i})\otimes\Delta^{\psi_i}$, for $i\in \{1,2,3\}$, related to each other as in Definition \ref{DefinitionWellBehaved}. Write each of the $\sigma_i$ as a pair $(\sigma_i^F,\sigma_i^P)$, where $\sigma_i^F\in F(\Delta^{\varphi_i})$ and $\sigma_i^P\in \Delta^{\psi_i}$. By construction of $F\otimes R(P)$, we have $\sigma^P=\sigma_i^P=\tau^P$, for all $1\leq i\leq 3$, and in particular, $\Delta^{\psi}\cap\Delta^{\psi_2}$ is not empty, write $\Delta^{\psi'}$ for this intersection. Furthermore, since $\Delta^{\varphi_1}$ contains both $\Delta^{\varphi}$ and $\Delta^{\varphi_2}$, there exists a smallest simplex of $N(P)$ - write it $\Delta^{\varphi'}$ - containing $\Delta^{\varphi}\cup \Delta^{\varphi_2}$. We then get the following factorization

\begin{equation*}
\begin{tikzcd}[column sep = small]
(\Delta^{\varphi},\Delta^{\psi})
&(\Delta^{\varphi_1},\Delta^{\psi_1})
\arrow{l}
\arrow{r}
\arrow{dr}
&(\Delta^{\varphi_2},\Delta^{\psi_2})
&(\Delta^{\varphi_3},\Delta^{\psi_3})
\arrow{l}
\arrow{r}
\arrow{dl}
&(\Delta^{\varphi},\Delta^{\psi})
\\
&&(\Delta^{\varphi'},\Delta^{\psi'})
\arrow{u}
\arrow{ull}
\arrow{urr}
\end{tikzcd}
\end{equation*}
Next, assume that we are given a family of $(\Delta^{\varphi_i},\Delta^{\psi_i})$, together with $\sigma_i\in F(\Delta^{\varphi_i})\otimes \Delta^{\psi_i}$, $0\leq i\leq 2n+2$, as in Definition \ref{DefinitionWellBehaved}. Then as earlier, write $\sigma_i=(\sigma_i^F,\sigma_i^P)$. One must have $\sigma_i^P=\sigma_j^P$, for all $0\leq i,j\leq 2n+2$, and so, the intersection $\cap_i \Delta^{\psi_i}$ must be non-empty, since it contains $\sigma_i^P$, for all $i$, write it $\Delta^{\psi'}$. Then, the intersection $\cap_i\Delta^{\varphi_i}$ is also non-empty since it contains $\Delta^{\psi'}$, write it $\Delta^{\varphi'}$. To obtain a diagram of shape (\ref{EquationAlmostFilteredPoset}), pick $d'_i= (\Delta^{\varphi_i},\Delta^{\psi'})$, and $e=(\Delta^{\varphi'},\Delta^{\psi'})$. As mentioned earlier, the $\sigma_i^P$ can already be seen as elements of $\Delta^{\psi'}$, so one can pick $x'_i=x_i$.
\end{proof}

\bibliographystyle{alpha}
\bibliography{biblio}

\end{document}